\documentclass{amsart}
\usepackage{amscd,latexsym, amsthm,amsfonts,amssymb,amsmath,amsxtra}
\usepackage[all]{xy}
\usepackage{setspace}
\usepackage{cases}
\usepackage{graphicx,epsfig}
\usepackage{extarrows}
\usepackage{mathrsfs}
\usepackage{enumerate}
\usepackage{eucal}
\usepackage{hyperref}
\usepackage{mathptmx}
\usepackage[square,sort&compress,comma,numbers]{natbib}
\usepackage{upgreek}
\usepackage{chemarrow}

\usepackage{mathtools}
\usepackage{tikz-cd}
\usetikzlibrary{decorations.pathmorphing}

\newcommand\xrsquigarrow[1]{%
	\mathrel{%
		\begin{tikzpicture}[%
			baseline={(current bounding box.south)}
			]
			\node[%
			,inner sep=.44ex
			,align=center
			] (tmp) {$\scriptstyle #1$};
			\path[%
			,draw,<-
			,decorate,decoration={%
				,zigzag
				,amplitude=0.7pt
				,segment length=1.2mm,pre length=3.5pt
			}
			] 
			(tmp.south east) -- (tmp.south west);
		\end{tikzpicture}
	}
}

\renewcommand{\eta}{\upeta}
\renewcommand{\tau}{\uptau}

\newcommand{\Mbar}{\bar{M}}
\newcommand{\Msig}{\ensuremath{M^{\upsigma}}}
\newcommand{\Mbarsig}{\bar{M}^{\upsigma}} 
\newcommand{\MbarO}{\bar{M}_{0}}
\newcommand{\MbarI}{\bar{M}^{1}}
\newcommand{\MbarIsig}{\bar{M}^{1,\upsigma}}

\newcommand{\FlG}{G_{\upkappa}/U_{-}^{\sig}}
\newcommand{\Ds}{\mathbb{D}^*}
\def\smallint{\begingroup\textstyle \int\endgroup}

\newcommand{\Xsig}{\ensuremath{X^{\upsigma}}}

\newcommand{\xflat}{x^{\flat}}
\newcommand{\xbarflat}{\bar{x}^{\flat}}
\newcommand{\yflat}{y^{\flat}}

\newcommand{\thet}{\uptheta}

\newcommand{\thetxbar}{\uptheta_{\bar{x}}}

\newcommand{\an}{\ensuremath{\mathrm{an}}}

\newcommand{\bet}{{\upbeta}}
\newcommand{\betx}{\upbeta_x}
\newcommand{\bety}{\upbeta_y}
\newcommand{\betxbar}{{\upbeta}_{\bar{x}}}
\newcommand{\betxsig}{{\upbeta_{x}^{\upsigma}}}
\newcommand{\sigbetx}{{\upsigma(\upbeta_x)}}

\newcommand{\eps}{\upvarepsilon}

\newcommand{\rf}{\ensuremath{\mathrm f}}
\newcommand{\rfxflat}{\ensuremath{\mathrm{f}_{x^{\flat}}}}

\def\smallint{\begingroup\textstyle \int\endgroup}
\newcommand{\Intxflat}{\ensuremath{\smallint_{x^{\flat}}}}
\newcommand{\Intxflatbar}{\ensuremath{\overline{\smallint_{x^{\flat}}}}}
\newcommand{\Intyflat}{\ensuremath{\smallint_{y^{\flat}}}}
\newcommand{\Intyflatbar}{\ensuremath{\overline{\smallint_{y^{\flat}}}}}

\newcommand{\II}{\ensuremath{\mathbb I}}
\newcommand{\IIplus}{\ensuremath{{\mathbb I}_+}}

\newcommand{\gam}{\upgamma}

\newcommand{\Gamxflat}{\Gamma_{x^{\flat}}}

\newcommand{\Gamxflatbar}{\overline {\Gamma_{x^{\flat}}}}

\newcommand{\gamxbarflat}{\upgamma_{\bar{x}^\flat}}

\newcommand{\lmd}{{\uplambda}}
\newcommand{\Lmd}{{\Lambda}}
\newcommand{\Lams}{\Lambda^*}
\newcommand{\LamsI}{\Lambda^{*,1}}
\newcommand{\LamsO}{\Lambda^{*,0}}
\newcommand{\GLlamstar}{{\GL(\Lambda_{\upkappa}^{*})}}

\newcommand{\zet}{\upzeta}
\newcommand{\zetatil}{\tilde{\upzeta}}
\newcommand{\sZzeta}{\mathcal{Z}_{\upzeta}}
\newcommand{\sZeta}{\mathcal{Z}_{\upeta}}

\newcommand{\etatil}{\tilde{\upeta}}

\newcommand{\kap}{{\upkappa}}
\newcommand{\Gka}{G_{\upkappa}}

\newcommand{\Gkasig}{G_{\upkappa}^{\upsigma}}

\newcommand{\gmka}{\mathbb{G}_{m,\upkappa}}

\newcommand{\Gmka}{\mathbb{G}_{m,\upkappa}}

\renewcommand{\mu}{\upmu}
\newcommand{\mutil}{{\tilde{\upmu}}}
\newcommand{\sigmu}{{\upmu^{\upsigma}}}
\newcommand{\sigmutil}{{\upsigma(\tilde{\upmu})}}
\newcommand{\mutilsig}{{\tilde{\upmu}^{\upsigma}}}

\newcommand{\upmuka}{{\upmu: \mathbb{G}_{m,\upkappa}\to G_{\upkappa}}}

\newcommand{\rI}{\ensuremath{\mathrm I}}

\newcommand{\sA}{\ensuremath{\mathcal A}}

\newcommand{\sD}{\ensuremath{\mathcal D}}

\newcommand{\sG}{\ensuremath{\mathcal G}}
\newcommand{\sH}{\ensuremath{\mathcal H}}

\newcommand{\sK}{\ensuremath{\mathcal K}}

\newcommand{\sM}{\ensuremath{\mathcal M}}

\newcommand{\sO}{\ensuremath{\mathcal O}}
\newcommand{\sP}{\ensuremath{\mathcal P}}

\newcommand{\sS}{\ensuremath{\mathcal S}}

\newcommand{\sU}{\ensuremath{\mathcal U}}
\newcommand{\sV}{\ensuremath{\mathcal V}}

\newcommand{\sZ}{\ensuremath{\mathcal Z}}

\newcommand{\sig}{{\upsigma}}
\newcommand{\udl}{\underline}

\newcommand{\iplus}{{\mathrm{I}_+}}

\newcommand{\Pplu}{P_+}

\newcommand{\pplu}{p_+}
\newcommand{\pmin}{p_{-}}
\newcommand{\Pminsig}{P_{-}^{\upsigma}}

\newcommand{\uplu}{u_+}

\newcommand{\umin}{u_-}
\newcommand{\Umin}{U_-}
\newcommand{\Uminsig}{{U_-^{\upsigma}}}
\newcommand{\barR}{{\bar{R}}}
\newcommand{\bari}{\bar{R}_i}

\newcommand{\Zp}{{\mathbb{Z}_p}}
\newcommand{\Fp}{\mathbb{F}_p}
\newcommand{\Fpbar}{\bar{\mathbb F}_p}

\newcommand{\pinf}{p^{\infty}}

\newcommand{\rg}{{\mathrm{g}}}

\newcommand{\gmoduminsig}{{G_{\upkappa}/\Uminsig}}
\newcommand{\underI}{{\underline{\rm I}}}
\newcommand{\Imin}{{{\rm I}_{-}}}
\newcommand{\Iplusig}{{\rm I}_{+}^{\upsigma}}
\newcommand{\GmodUminsig}{{G_{\upkappa}/U_{-}^{\upsigma}}}

\newcommand{\wk}{\ensuremath{W(k)}}
\newcommand{\wkap}{\ensuremath{W(\upkappa)}}

\newcommand{\hra}{{\hookrightarrow}}

\newcommand{\emu}{{E_{\upmu}}}

\newcommand{\too}{{\longrightarrow}}

\newcommand{\zipstack}{{[G_{\upkappa}/\emu]}}
\newcommand{\zipstackI}{\ensuremath{[(G_{\upkappa}/U_{-}^{\upsigma})/P_+]}}

\newcommand{\Gzips}{\ensuremath{{G\textsf{-Zip}^{\upmu}}}}

\newcommand{\Uplusig}{{U_+^{\sig}}}

\newcommand{\xbarleftflat}{{{}^{\flat}\bar{x}}}

\newcommand{\Rbar}{\bar{R}}
\newcommand{\Rbarperf}{\bar{R}_{\mathrm{perf}}}
\newcommand{\Abar}{\bar{A}}

\newcommand{\xbar}{\bar{x}}

\newcommand{\teta}{\uptheta}

\newcommand{\sAxbar}{\sA_{\bar{x}}}

\newcommand{\HIdR}{\ensuremath{\mathrm{H}^1_{\rm dR}}}

\newcommand{\C}{\ensuremath{\mathbb C}}

\newcommand{\Z}{\ensuremath{\mathbb Z}}

\newcommand{\Afp}{\ensuremath{\mathbb{A}_f}^p}

\newcommand{\cancong}{\ensuremath{{\ \displaystyle \mathop{\cong}^{\text{\tiny{can}}}}\ }}

\newcommand{\Ch}{{\mathrm{Ch}}}

\newcommand{\cris}{{\mathrm{cris}}}

\newcommand{\can}{\mathrm{can}}

\newcommand{\dR}{{\mathrm{dR}}}

\newcommand{\gr}{{\mathrm{gr}}}

\newcommand{\GL}{{\mathrm{GL}}}

\newcommand{\GSp}{{\mathrm{GSp}}}

\newcommand{\Hom}{{\mathrm{Hom}}}

\newcommand{\id}{{\mathrm{id}}}
\renewcommand{\Im}{{\mathrm{Im}}}

\newcommand{\Isom}{{\mathrm{Isom}}}

\newcommand{\Ker}{{\mathrm{Ker}}}

\newcommand{\q}{\ensuremath{\mathbb{Q}}}

\newcommand{\qp}{\ensuremath{{\mathbb Q}_p}}

\newcommand{\pr}{{\mathrm{pr}}}

\renewcommand{\mod}{\ \mathrm{mod}\ }

\newcommand{\reg}{{\mathrm{reg}}}
\newcommand{\Res}{{\mathrm{Res}}}

\newcommand{\Sh}{{\mathrm{Sh}}}

\newcommand{\Spec}{{\mathrm{Spec}}}
\newcommand{\Spf}{{\mathrm{Spf}}}

\newcommand{\Zhat}{\hat{\mathbb{Z}}}

\newcommand{\FF}{\mathrm{F}}
\newcommand{\FFbar}{\bar{\mathrm{F}}}
\newcommand{\VV}{\mathrm{V}}
\newcommand{\VVbar}{\bar{\mathrm{V}}}
\newcommand{\HH}{{\mathrm{H}}}

\newcommand{\sbt}{\subseteq}
\newcommand{\spt}{\supseteq}

\DeclareMathAlphabet{\pazocal}{OMS}{zplm}{m}{n}

\newcommand{\perf}{\ensuremath{\mathrm{perf}}}

\numberwithin{equation}{subsection}
\newcommand{\K}{\ensuremath{\mathsf{K}}}
\newcommand{\Ktil}{\ensuremath{\tilde{\mathsf{K}}}}
\newcommand{\sk}{\ensuremath{\mathcal{S}_{\mathsf{K}}}}

\newcommand{\shk}{\ensuremath{\mathrm{Sh}_{\mathsf{K}}}}
\newcommand{\shkc}{\mathrm{Sh}_{\mathsf{K}, \mathbb{C}}}
\newcommand{\shkan}{\mathrm{Sh}_{\mathsf{K}}^{\mathrm{an}}}
\newcommand{\shkcan}{\mathrm{Sh}_{\mathsf{K}, \mathbb{C}}^{\mathrm{an}}}

\newcommand{\RI}{\mathbf{R}^1}

\theoremstyle{remark}
\newtheorem{thm}{\rm{\textbf{Theorem}}}[subsection]

\newtheorem{lem}[thm]{\rm{\textbf{Lemma}}}
\newtheorem{prop}[thm]{\rm{\textbf{\textbf{Proposition}}}}

\newtheorem{defn}[thm]{\rm{\textbf{Definition}}}
\newtheorem{ex}[thm]{\rm{\textbf{Example}}}
\newtheorem{rmk}[thm]{\rm{\textbf{Remark}}}

\author[Q.~Yan]{Qijun Yan}
\title[An alternative construction of zip period maps]{An alternative construction of zip period maps for Shimura varieties}
\address{Morningside Center of Mathematics, Chinese Academy of Sciences, Beijing, 100\,190 China}
\email{yanqmath\symbol{64}amss.ac.cn}
\date{}
\begin{document}
	
	\begin{abstract}
		Let $ S $ be the special fibre of a Shimura variety of Hodge type, with good reduction at a place above $p$. We give an alternative construction of the zip period map for $S$, that is used to define the Ekedahl-Oort strata of $ S $. The method employed is local, $p$-adic,  and group-theoretic in nature.	
	\end{abstract}
	
	\maketitle

	\section{Introduction}
	\subsection{History of zip period map}

	The \textbf{zip period map} in the title arises in the development of Ekedahl-Oort (EO, for simplicity) stratification theory for Shimura varieties. Initially, the EO stratification was defined by Ekedahl and Oort \cite{OortAstratificationofModuliSpaces} for the moduli space of principally polarized abelian varieties $\sA_g\otimes\Fp$ of dimension $g$ in characteristic $p>0$ (can be viewed as the Siegel type Shimura variety), by declaring that two points $(A, \lmd)$ and $(A', \lmd')$ over $\bar{\mathbb{F}}_p$ lie in the same stratum if their $p$-kernels are isomorphic.
	
	Later on, this stratification was extended  to PEL type Shimura varieties in  \cite{GorenOort} \cite{MoonenGroupschemeswithadditionalstructure},\cite{MoonenDimensionformula}, \cite{Moonen&WedhornDiscreteinvariants}, \cite{ViehmannWedhornEOPELtype}, and to Hodge type Shimura varieties in \cite{VasiuModpFcrystal}, \cite{ChaoZhangEOStratification}.  The underlying idea is the same as in the Siegel case, that is, considering isomorphism classes of  $p$-kernels of abelian varieties with additional structures. The way of defining and studying these strata evolves over time. Let $S$ be the special fibre of a PEL type Shimura variety of good reduction at $p$; it is defined over a finite field, say $\kap$. In order to give the dimension formula for the EO strata of $S$, Wedhorn  \cite{WedhornDimensionofOortStrata} constructed a sequence of morphisms of stacks over $\kap$ (later viewed as a period map in characteristic $p$)
	\begin{equation}\label{Wedhmap}
	 S \too  \mathsf{BT}_{1}\to \mathsf{DS}_1,
	\end{equation}
	where $\mathsf{BT}_1$ is the stack of BT 1's (i.e., the $p$-kernel of $p$-divisible groups) with PEL structure and $\mathsf{DS}_1$ is the stack of Dieudonn\'e spaces with PEL structures (i.e., Dieudonn\'e modules associated with BT 1's with extra structure). He shows that the map $S\to  \mathsf{BT}_{1}$ is smooth and the natural map $\mathsf{BT}_{1}\to \mathsf{DS}_1$  given by the crystalline Dieudonn\'e functor is a homeomorphism. 
	
   
   Soon, Moonen and Wedhorn  \cite{Moonen&WedhornDiscreteinvariants} established the theory of $F$-zips with the underlying idea that an $F$-zip structure on a vector bundle in characteristic $p$ is like a Hodge structure in characteristic $0$. Moreover, they constructed a morphism of $\kap$-stacks 
  \begin{equation}\label{MoonWedmap}
  S\too [G\backslash X_{\mu}],
  \end{equation} 
where $X_{\mu}$ is the moduli of trivialized $F$-zips with PEL structures (of certain type $\mu$ determined by $S$). Here the map is given by taking the $F$-zip associated with the universal BT 1 over $S$, which by definition is the de Rham cohomology $\HH_{\dR}^1(\sA/S)$ of the universal abelian scheme $\sA$, equipped with its $F$-zip structure. 
In fact, an $F$-zip associated with a BT 1 is equivalent to the corresponding Dieudonn\'e space defined in \cite{WedhornDimensionofOortStrata} and hence the map  \eqref{MoonWedmap} is essentially the same as the map  \eqref{Wedhmap}. Thanks to the analogy of $F$-zip structures to Hodge structures, the map in \eqref{MoonWedmap} is considered as a period map in characteristic $p$. 

Based on the theory of $F$-zips, Pink, Wedhorn and Ziegler \cite{PinkWedhornZiegler2} defined the notion of $G$-zips (as $F$-zips with $G$-structures) and the stack of $G$-zips of type $\mu$, denoted by $\Gzips$; we refer to \S \ref{S: G-zips} for its precise definition.
In the meanwhile, they show that the stack $\Gzips$ can be realized as the quotient stack of $G$ by some zip group $\emu$ (a notion defined in \cite{PinkWedhornZiegler1}), i.e., we have an isomorphism of $\kap$-stacks (again see \S \ref{S: G-zips})
 \begin{equation}\label{QuotZipIsom}
  \Gzips\cong \zipstack. 
\end{equation}
Suppose now that  $S$ is of Hodge type and $p\geq3$. 
 In order to extend EO stratification to Shimura varieties of Hodge type,  Zhang~\cite{ChaoZhangEOStratification} (see also \cite{Wortmann}) constructed a map of algebraic stacks (see \S \ref{SectionDefinitonofEO} for a review of construction)
 \begin{align}\label{Chao'szeta}
 	\upzeta: S\longrightarrow \Gzips,
 \end{align}
 and showed that $\zet$ is smooth.  The EO strata of $S$ are defined as geometric fibres of $\zet$. The strata thus defined are automatically smooth and many properties on these strata are obtained by translating the information of the target stack into that of $S$ via $\zet$; see loc. cit. for details.
 
  We call map $\zet$ the \textbf{zip period map} for $S$. There are also some other variants of this map: for example, the perfectly smooth map  $\Sh_{\mu}\to \mathrm{Sht}_{\mu}^{\mathrm{loc}(2,1)}$ in \cite[Rem. 7.2.5]{XiaoZhuCycle} (see also \cite{SYZEKOR} for its generalization) and the map $\eta: S\to \sD_1/\sK^{\diamond}$ in \cite[Thm. 8.5.2]{Yan18}. 
 	The aim in this paper is to give an alternative  construction of  $\zet$ (more precisely, the composition of $\zet$ with the isomorphism in \eqref{QuotZipIsom}) avoiding the language of $G$-zips, which (we hope) provides a different perspective of understanding the already existing zip period map.

	\subsection{Main results and the strategy of proof}
	\noindent 
	 Let $(\textbf{G}, \textbf{X})$ be a Shimura datum of Hodge type and denote by  $\sk$ the Kisin-Vasiu integral model of the associated Shimura variety $\Sh_{\K}(\textbf{G}, \textbf{X})$ of level $\K$ which is hyperspecial at $p$. This hyperspecial condition on $\K$ implies that $\textbf{G}_{\qp}$ admits a reductive $\Zp$-model $\sG$, whose special fibre we denote by $G$. Recall that the integral model $\sk$ is a quasi-projective and smooth scheme over  $ \mathcal{O} $, the localization at some place above $p$ of the ring of integers of the reflex field  of $(\textbf{G}, \textbf{X})$. Thanks to the hyperspecial assumption $\sO$ is unramified at $p$. Write  $\kap$ for the residue field of $\sO$ and  $S:=\sk\otimes_{\sO} \kap $. Let $\upmu: \gmka\to \Gka$ be a representative for the reduction over $ \kap $ of  the  $ \textbf{G}(\mathbb{C}) $-conjugacy class $[\upmu]_{\mathbb{C}} $ of the inverses of Hodge cocharacters $ \mathbb{G}_{m, \C}\to \textbf{G}_{\mathbb{C}} $ determined by $ (\textbf{G}, \textbf{X}) $.	
	 
	Denote by $ P_{\pm}\sbt \Gka$ the opposite parabolic subgroups of $\Gka$ defined by $ \mu $ and  $U_{\pm}\sbt P_{\pm}$ the corresponding unipotent radicals, and $\Uminsig$ the base change of $\Umin$ along the $p$-power Frobenius $\sig: \kap\to \kap$; the same convention applies to other notations of the form $(\cdot)^{\sig}$.
	The stack $\Gzips$  is in fact isomorphic to some quotient stack $\zipstack$ of $ \Gka $ by the smooth algebraic group $ E_{\upmu}=\Pplu\ltimes \Uminsig $ (see \S \ref{S: G-zips} for the action),  but such an isomorphism is not quite formal. The following nearly trivial observation turns out to be important to this work: since $ \Uminsig $, as a normal subgroup of $ \emu$, acts freely on $ \Gka $ by right multiplication, by passing to quotient we obtain a canonical isomorphism of algebraic stacks over $\kap$ (\S \ref{S: ZipMapEta}): \[\zipstackI \cong \zipstack,\] where $\FlG$ is represented by a smooth $ \kap $-scheme. Hence to give the zip period map $\zet$ above is equivalent to a give a $P_+$-torsor, say $T$,  over $S$ and a $P_+$-equivariant map of $\kap$-schemes $T\to \GmodUminsig$.  The natural candidate for $T$ is the  scheme $\rI_+$ of trivializations of the Hodge filtration $\HIdR(\sA/S)\spt \omega_{\sA/S}$, respecting certain tensors that we do not specify in this introduction. The torsor $\rI_+$ is part of the datum for the universal $G$-zip which  defines~$\zet$. 
	
		\begin{thm}[Thm. \ref{MainThm1}, Thm. \ref{MainThm2}]\label{Constrgam}\begin{enumerate}[(1).]
			\item There exists an (explicitly constructed) morphism of $\kap$-schemes 
			\[\gam: \rI_+\to \FlG.\] 
			\item The map  $ \gam $ is $ P_{+} $-equivariant, and hence induces a morphism of algebraic $\kap$-stacks 
			\begin{equation*}\label{map zeta prime}
				\upeta: S\longrightarrow \zipstackI \cong \zipstack \cancong \Gzips.
			\end{equation*}
		\end{enumerate}
	\end{thm}
		\begin{thm}[Thm. \ref{CompThm}]\label{LocalGlobalCompat}
		We have a natural 2-isomorphism $\upeta\cong \upzeta$. Consequently, we give an alternative construction of the zip period map for $ S $. 
	\end{thm}
	We describe now the construction of $\gam$ on geometric points. Take $k=\bar{\mathbb{F}}_p$.
	From now on we fix a cocharacter $ \tilde{\upmu}: \mathbb{G}_{m, W(\kap)}\to \sG_{\wkap} $ of $\sG_{\wkap}$ which lifts $\upmuka$. A point 
	$ \xbarflat=(\xbar, \betxbar) \in \rI_+(k)$ 
	consists of a point $ \xbar\in S(k)$ and a trivialization 
	\[ \betxbar : [\Lmd^{*}_{k}\spt \Lmd^{*,1}_{k}]\cong [\HH^{1}_{\dR}(\sA_x/k) \supseteq \omega_{\sA_{\xbar}/k}]\cong [\Mbar\\ \supset \Mbar_1], \] 
	respecting tensors on both sides, where  $ \Mbar$ and $  \Mbar_1$  denote the reduction modulo $p$ of  the contravariant Dieudonn\'e module $M$ of the $ p $-divisible group $ \sA_{\xbar}[\pinf] $ over $ k $ and its Hodge filtration  respectively $M_1$ (\S \ref{BT/R}). Here $\Lmd^{*,1}_{k}$ is the weight $ 1 $ subspace of $ \Lmd^{*}_{k} $ induced by  $ \upmu_k: \mathbb{G}_{m, k}\to G_{k}$.  Let $\II_+$ be the integral model over $\sS$ of $\rI_+$. The first step of the construction of $\gam$ on $k$-points is to choose a lift $ \xflat =(x, \betx)\in \II_+ (W(k))$ of $\xbarflat$ which provides a lift of~$\betxbar$,
	\[ \betx: [\Lmd^{*}_{W(k)}\spt \Lmd^{*,1}_{W(k)}]\cong [\HH^{1}_{\dR}(\sA_x/W(k)) \supseteq \omega_{\sA_x/W(k)}], \] and hence a trivialization of $M$ via the canonical isomorphism $\HH^{1}_{\dR}(\sA_x/W(k))\cong M$, and then show that via the trivialization $\betx$ the Frobenius of $M$ admits a uniform decomposition
	\[ \Intxflat   \mutil_{W(k)}^{\sig}(p) \mathrm\ \ \ \text{with}\ \ \Intxflat \in \sG(W(k))\sbt \GL(\Lmd^{*}_{W(k)}), \]
	where $ \mutil_{W(k)}^{\sig}: \mathbb{G}_{m, W(k)}\to \sG_{W(k)}^{\sig} \cong \sG_{W(k)}$ is the base change along $ \sig: W(k) \to W(k) $ of $ \mutil_{W(k)} $.
	The key point here is that the element $\Intxflat$ is integral and hence we can take its reduction modulo $p$, denoted by $ \overline{ \Intxflat}  \in G(k) $.  
	Then one proceeds by showing that the image of $ \overline{ \Intxflat} $ in $  \GmodUminsig(k) $ is independent of lifts $ \xflat $ (Lem. \ref{Glue1}); we denote it by $\gamxbarflat$. To summarize, the map $\gam$ on $k$-points is given by 
 performing the following operations (\S \ref{S:ConstructionofEta})
	\begin{equation}\label{operations}
		\xbarflat\in \iplus(k)  \xrsquigarrow{\text{choose } x^{\flat}}\Intxflat\in \sG(W(k)) \xrsquigarrow{\mod p}\Intxflatbar\in G(k) \xrsquigarrow{\text{projection}} \gamxbarflat\in G/\Uminsig(k).
	\end{equation}

	\noindent


	 The technical heart of the construction of $\gam$  in Thm. \ref{Constrgam} is to justify the operations in \eqref{operations} and to show that these operations can be performed in a relative sense: for every smooth $\upkappa$-algebra $\Rbar$ which (automatically) admits a simple frame (equivalently, a crystalline prism if one prefers) and every point $\xbarflat\in \rI_+(\Rbar)$, we can construct a point $\gamxbarflat\in \gmoduminsig(\Rbar)$ whose specialization at geometric points coincides with \eqref{operations}; see Prop. \ref{Locconstr}. This relative construction relies on relative classifications of $p$-divisible groups as in \cite{deJong95} and is more complicated in the sense that in the relative setting we need to compare not only different lifts $ \xflat $ as aforementioned, but also different choices of simple frames for $ \bar{R} $. The independence of these two different types of choices are proved via matrices calculations; see \S \ref{LocConstr}. Finally the global map $\gam: \rI_+\to \FlG$ is obtained by first constructing it on Zariski opens of $\rI_+$ and then gluing the local maps together. 
	  

	We now give some comments on the comparison of $\zet$ and $\upeta$. Zhang's construction of $ \zeta $ uses the global geometry over characteristic $p$, namely the language of $G$-zips (which is somewhat complicated: for example, a $G$-torsor involves three torsors plus some delicate zip relations) but follows in spirit the original intuitive definition of EO stratification for $ \sA_g $ since the stack $ \Gzips $ can be viewed as the moduli space of BT 1's while the universal $ G $-zip for $ \zeta $ corresponds to $\sA[p]$, the universal BT 1 over $ S $. In particular $ \zeta $ is determined by  $\sA[p] $. In contrast, the construction of $\eta$ is local and group-theoretic; it avoids the fancy language of $G$-zips and uses only one torsor, $\rI_+$. But since it does not start with $\sA[p]$, in the end the dependence of  $ \gam $ (hence of $\eta$) on $ \sA[p] $ is obscured. From our local construction one sees better  the role that the zip group $\emu=P_+\ltimes \Uminsig$ plays in the business of zip period map.  For example,  given a $k$-point $\xbarflat$ of the $P_+$-torsor $\rI_+$, different lifts $\xflat$ produce the same $\Uminsig$-coset in $G(k)$; this coincides with Faltings deformation theory which says that the deformation of the $p$-divisible group $\sAxbar[\pinf]$ is controlled by the integral model of $\Umin$. The proof of Thm. \ref{LocalGlobalCompat} is not formal, partly because the canonical isomorphism $\zipstack \cancong \Gzips$ is not formal. 
	
	This work has a certain amount of overlap (not on main results) with my PhD thesis \cite{Yan18}. The connection between these two works will be made in a subsequent paper. 
	
	\subsection{Notational convention}\label{NotaConven}
	Throughout the paper we fix a prime number $p\geq 3$. The  Dieudonn\'e crystals (resp. modules) used in this paper are contravariant. 	Let $ R $ be a ring and $ M $ an $ R $-module.  If $ \sig: R \to R$ is a ring endomorphism  we write $M^{\sig}=\sig^{*}M$ for the base change $M\otimes_{R, \sig}R$. 	If $ M  $ is finite locally free, we denote by $ M^{*} $ its  dual $ R $-module. Then we have the canonical identification $M^{\otimes}\cancong M^{*, \otimes}$ of $R$-modules, where $M^{\otimes}$ is the direct sum of all $R$-modules obtained from $M$ by applying the operations of taking duals, tensor products, symmetric powers and exterior powers. Here, as a general convention, the notation ``$ \cancong $" means canonical isomorphism between mathematical objects. For any $ R $-automorphism $ f: M\cong M$,  we have an induced isomorphism $ (f^{-1})^{*}: M^{*}\to M^{*}, a\mapsto f^{-1}\circ a $, and hence a canonical isomorphism of $ R $-group schemes $(\cdot)^{\vee}: \GL(M)\cancong \GL(M^*), \  g\longmapsto g^{\vee}:=(g^{-1})^{*}$. We also use the letter $M$ to denote an Levi subgroup (of some algebraic group) but it shall be always clear from the context whether $M$ is a module or an algebraic group.  The decoration $\bar{( \ )}$ usually indicates that the object in question is over the characteristic $p$ world or is the reduction modulo $p$ of $(\ )$; it shall be clear if it has some other meaning.
	
	For an $ \Fp $-algebra $ \Rbar $, we use $ \sig: \Rbar \to \Rbar $ for the absolute (i.e., $p$-power) Frobenius of $ \Rbar $. If $ X $ is a scheme over $\Rbar$, we write $ X^{\sig} $ for its pull back along $\sig$ and $\sig: X\to X^{\sig}$ the relative Frobenius over $\Rbar$. In particular, when $X$ is defined over $\Fp$, sometimes we also write $\sig: X\to X$ for the composition of the relative Frobenius of $X$ with the canonical isomorphism $\sig: \Xsig\cancong X$.  Similarly, if $f: X\to Y$ is a map between objects over $\Rbar$, we write $f^\sig$ for its base change along $\sig: \Rbar\to \Rbar$.  
	Now let $ k $ be a perfect field and $ \mathcal{G} $ a group scheme over $ W(k) $, which is defined over $ \mathbb{Z}_p $. For a  $ W(k) $-algebra  $ R $ with a Frobenius lift $\sig= \sig_R: R\to R $  over $ W(k) $, we often denote by 
	\begin{equation}\label{MixFrob}
		\sig: \mathcal{G}(R)\to \mathcal{G}(R)
	\end{equation} 
	the homomorphism induced by $ \sig: R\to R $ (note that $ \sig: R\to R $ is only a $ \mathbb{Z}_p $-endomorphism, but not a $ W(k) $-endomorphism in general). We abuse language and call also \eqref{MixFrob}  ``Frobenius" of $\sG$. In case $\sG$ is defined over $\Fp$, this Frobenius coincides with relative Frobenius mentioned above.

	In this paper, for quotient stacks we systematically use right actions instead of left or mixed actions; for example, the stack $\zipstack$ in this paper corresponds to $[\emu\backslash \Gka]$ in the literature.	

\subsection{Acknowledgements}
The main idea of this work has its origin in my PhD thesis \cite{Yan18} during the course of which I received  help from many people, including Fabrizio  Andreatta,  Bas Edixhoven, Ulrich Goertz, Bart de Smit and Torsten Wedhorn; their help continues to  contribute in the context of this work. Hence, it is my pleasure to express my gratitude to all of them once again. I  thank Chao Zhang for his encouragement  and for answering many of my questions over the years.  I  thank Liang Xiao for suggestions on the writing of this paper and answering many questions on topics of Shimura varieties. I thank Lei Fu for patiently explaining the details of local systems and other things which has indirectly contributed to this work. While working on this paper, I moved from Yau Mathematical Sciences Center, Tsinghua University  to the Morningside Center of Mathematics, Chinese Academy of Sciences. I thank both of these institutions as well as my mentors Zongbin Chen and Xu Shen for their support.

   \section{Classification of \texorpdfstring{$p$}{Lg}-divisible groups (recollection)}\label{S:pdivisible}
	Throughout this section we let $ k $ be a perfect field of characteristic $ p$ and denote by $ \sig: W(k) \to W(k)$ its  unique ring automorphism inducing the absolute Frobenius of $ k $. 
	\subsection{Existence of simple frames} \label{SettingofRings}    
	\begin{lem} \label{Froblifts}
		Let $\barR$ be a $k$ algebra which Zariski locally admits a finite $p$-basis (\cite [Def. 1.1.1]{deJong95}, or  \cite[Def. 1.1.1]{BWDieudonneCrystallineIII}). The following holds:
		\begin{enumerate}[(1)]
			
			\item There exists a $p$-complete flat $W(k)$-algebra $R$ lifting $\barR$ (i.e., $R/pR\cong \barR$), which is formally smooth over $W(k)$ with respect to the $p$-adic topology. Such an $R$ is unique up to (nonunique) isomorphisms and we call it a \textbf{lift of} $ \barR $. 
			\item There is a ring endomorphism $\sig=\sig_R: R\to R $ lifting the absolute Frobenius of $\barR$, which is compatible with $ \sig: W(k)\to W(k) $. We  call it a \textbf{Frobenius lift }of $ R $ over $ W(k) $.
			\item Let $ \barR, R $ be as above and $ \bar{A}$ an \'etale $ \bar{R} $ algebra. Then there exists a formally \'etale $ R $-algebra $ A $ (for the $ p $-adic topology), unique up to unique isomorphism, such that $ A $ lifts $ \bar{A} $ and the structure ring homomorphism $ R \to A$ lifts the structure homomorphism $\bar{R}\to \bar{A}$. Moreover, every Frobenius lift $\sig_R: R\to R$ of $ R $ over $ W(k) $ extends uniquely to a Frobenius lift $\sig_A: A\to A$ of $A$ over $ W(k) $. 
			\item Let $(R, \sig)$ be as above. If $\mathfrak{m}$ is a maximal idea of $R$, then $\sig$ extends uniquely to a Frobenius lift of the $\mathfrak{m}$-adic completion $ \widehat{R}_{\mathfrak{m}} $ of $R$, which is a lift of the $\mathfrak{m}$-adic completion of $\bar{R}$. 
		\end{enumerate}
		
	\end{lem}

	\begin{proof}
		(1) and (2) is a special case of  \cite[Lem. 2.1]{KimWansuRelative} (take $I=(p)$) and (3) is a special case of  the first part of \cite[Lem.  2.5 ]{KimWansuRelative}.	For (4), note first that  $ \sig (\mathfrak{m}) \sbt  \mathfrak{m} $. This follows from the fact that $ \mathfrak{m} $ contains $ p $, and the fact that the morphism $ \Spec \bar{R}\to \Spec \bar{R}$ induced by the absolute Frobenius of $ \bar{R} $ is identity on topological spaces. Hence we can define $ \sig_{\widehat{R}_{\mathfrak{m}}}: \widehat{R}_{\mathfrak{m}}  \to \widehat{R}_{\mathfrak{m}} $ by sending an element $ (r_i)_{i} \in \varprojlim\limits_{i}R/\mathfrak{m}^{i}=\widehat{R}_{\mathfrak{m}}$ to $ (\sig(r_i))_{i}\in \widehat{R}_{\mathfrak{m}} $. 
	\end{proof}

\begin{ex}[{\cite[1.1.2]{BWDieudonneCrystallineIII}}] \label{ExofSimpFram} The main examples of $\Rbar$ in our later applications are: \begin{enumerate}[(1)]
	\item $\Rbar$ is a perfect $k$-algebra (the empty $p$-basis case). In this case, the unique simple frame of $\Rbar$ (up to unique isomorphism) is given by $(W(\Rbar), \sig)$. 
	\item $\Rbar$ is a smooth $k$-algebra of finite type. Here Zariski locally $\Rbar$ indeed admits a finite $p$-basis: Zariski locally $\Rbar$ is \'etale over  some polynomial algebra $\bar{A}=k[x_1,\cdots, x_n]$ which has the standard $p$-basis $ \{x_1, \cdots, x_n\} $; then the image of this $p$-basis in $\Rbar$ form a $p$-basis of $\Rbar$, as the relative Frobenius map $\bar{A}\otimes_{\sig, \bar{A}}\Rbar\to \Rbar, \ a\otimes r\mapsto ar^p$ is an isomorphism (hence the Frobenius $\sig: \Rbar\to \Rbar$ can be identified with the canonical ring map $\Rbar\to \bar{A}\otimes_{\sig, \bar{A}}\Rbar, r\mapsto 1\otimes r$). 
	\end{enumerate}
\end{ex}

   \begin{defn}\label{Def: simple frame}
		Let $\barR$  be as in Lem. \ref{Froblifts}. 
		A \textbf{simple frame} of $\barR$, relative to $W(k)$, is a pair $\underline{R}=(R, \sig)$, where $R$ is a  lift of $\bar{R}$ and $\sig: R\to R$ is a Frobenius lift of $R$.  
	\end{defn} 

     \begin{rmk}\label{Rmk: FramPrism}
     A simple frame $(R,\sig)$ over $\wk$ of $\Rbar$ is the same thing as a crystalline prismatic prism over the base prism $(W(k), \sig)$ in the sense of Bhatt-Scholze \cite{PrismBS}. Hence in fancier language, simple frames of $\Rbar$ should perhaps be termed as  \textbf{(crystalline) prismatic charts} of $\Rbar$.
     \end{rmk}

	\subsection{Classification of $ p $-divisible groups over $ \bar{R} $}\label{S:BT/Rbar} 
	
		Let $\barR$  be as in Lem. \ref{Froblifts} and $\underline{R}=(R, \sig)$ a simple frame of $\Rbar$. Till the end of this section, we assume further that \textbf{$\Rbar$ is as in Exam. \ref{ExofSimpFram}.}  As a preparation for later sections, we review in this subsection results on classification of $p$-divisible groups over  $\bar R$ (and over $R$ in the next subsection \S \ref{BT/R}), in terms of linear data over the simple frame $(R, \sig)$. 
		
		We denote by $\widehat{\Omega}_{R}$ the module of $ p $-adically continuous differentials of $ R $, i.e., \[ \widehat{\Omega}_{R}: = \varprojlim\limits_{n} \Omega_{(R/p^nR)/W(k)}^{1}.\]	It is a finite projective $ R $-module due to the finite $ p $-basis assumption on $ \barR $.    	We denote by 	$ \mathbf{DM}(\udl{R}, \nabla) $ 	the category of Dieudonn\'e modules with connections. Here a \textbf{Dieudonn\'e module with connection} (or simply a \textbf{Dieudonn\'e module}) over $\underline{R}$ (or simply over $R$ when $\sig$ is chosen) is a tuple 	$ (M, \FF, \VV, \nabla_{M}) $,	where $M$ is a finite locally free $R$-module and	$\FF: M^{\sig}\to M, \ \VV: M\to M^{\sig} $	are maps between $R$-modules such that 	\begin{equation} \label{P&P}	\FF\circ \VV=p\cdot \id_{M^{\sig}};\ \ \  \VV\circ \FF=p\cdot \id_M,	\end{equation}	and where 	$ \nabla_{M}: M\to M\otimes_{R}\widehat{\Omega}_{R} $	is an integrable topologically quasi-nilpotent connection over the $ p $-adically continuous derivation $ d_{R}: R\to \widehat{\Omega}_{R} $ of $ R $, with respect to which $ \FF $ is horizontal, i.e.,	$\nabla_{M} \circ \FF = (\FF\otimes \id_{\widehat{\Omega}_R})\circ \sig^*(\nabla_{M})$.

	For a $ p $-divisible group $ \bar{H} $ over $ \bar{R} $, we denote by $\mathbb{D}^*(\bar{H})$ the Dieudonn\'e crystal\footnote{The  superscript  $*$ in $\mathbb{D}^*(\bar{H})$ is used to indicate that our Dieudonn\'e crystal here is contravariant.} of $ \bar{H}$ as in \cite{BBM}, which coincides with the construction in \cite{MessingBT} up to duality: to be precise,  our 
	$\mathbb{D}^*(\bar{H})$ here corresponds to the Dieudonn\'e crystal $\mathbb{D}(\bar{H}^*)$ in \cite{MessingBT}, with $\bar{H}^*$ the dual $p$-divisible group of $\bar{H}$.  Following usual convention, we write
	$ \mathbb{D}^*(\bar{H})(R) $
	for the evaluation of $\mathbb{D}^*(\bar{H})$ at the canonical PD-thickening 
	$ R\twoheadrightarrow \barR $. 
	By functoriality of the formation of Dieudonn\'e crystals, we have 
	$\mathbb{D}^{*, \sig}(\bar{H})=\mathbb{D}^*(\bar{H}^{\sig})$,
	where $\mathbb{D}^{*, \sig}(\bar{H})$ is the pullback along 
	$\sig: \bar R \to \bar R$ of 
	$\mathbb{D}^{*}(\bar{H})$.
	Consequently we have canonical isomorphism 
	$\mathbb{D}^{*, \sig}(\bar{H})(R)\cong \mathbb{D}^{*}(\bar{H})^{\sig}$
	of $R$-modules. The Frobenius
	$\bar H \to {\bar H}^{\sig}$ and Verschiebung 
	$\bar H\to {\bar H}^{\sig}$ induces morphism of crystals,
	\[ 
	\FF: \mathbb{D}^*(\bar{H}^{\sig})\to \mathbb{D}^*(\bar{H}), \ \ 
	\VV: \mathbb{D}^*(\bar{H})\to \mathbb{D}^*(\bar{H}^{\sig}),
	\]
	such that 
	$\FF\circ \VV = p\cdot \id_{\mathbb{D}^*(\bar{H})}$ and 
	$\VV\circ \FF = p\cdot\id_{\mathbb{D}^*(\bar{H}^{\sig})}$.
	Evaluating at the thickening $R\twoheadrightarrow \
	\bar R$ we obtain $R$-linear maps $\FF, \VV$ for $\mathbb{D}^*(\bar{H})(R)$, just like an object in $ \mathbf{DM}(\udl{R}, \nabla) $ satisfying \eqref{P&P}. Denote by 
	$\big(\mathbf{BT}/\bar{R}\big)$ 
	the category of $p$-divisible groups over $\bar{R}$. The following classification result is known.
	
	\begin{rmk}\label{Rmk:Semi}
	 If $H=\bar{\sA}[p^{\infty}]$ for some abelian scheme $\bar \sA$ over $\bar R$ (the case we mostly concern for later applications), we have canonical isomorphism of Dieudonn\'e crystals (\cite[3.3.7, 2.5.6]{BBM}),																\[\mathbb{D}^*(\bar H)\cong \mathbb{D}^*(\bar \sA)\cong \RI\pi_{\rm CRIS, *}\sO_{\bar \sA}^{\rm cris}, \]  																						where $\pi: \bar \sA\to \Spec \bar R$ is the structure morphism. It follows then that we have  the following canonical isomorphism of $R$-modules, which is Frobenius equivariant	\begin{equation}	\HH^1_{\cris}(\bar \sA/R)\cong \mathbb{D}^{*}(\bar{H})( R).\end{equation}	
	\end{rmk}

	\begin{thm}\label{ClaBTmodp}
		For any $ p $-divisible group $ \bar{H} $ over $ \bar{R} $, there exists a natural connection 
		$ \nabla_{M}: M\to M\otimes_{R}\widehat{\Omega}_{R} $
		for $M=\mathbb{D}^*(\bar{H})(R)$ such that 
		the tuple 
		\begin{equation}
			\underline{M}=(M,\  \FF, \VV,  \nabla_{M})    
		\end{equation} 
		is an object in $\mathbf{DM}(\udl{R}, \nabla)$. Moreover, such an assignment gives an equivalence of categories between 
		$\big(\mathbf{BT}/\bar{R}\big)$ and $ \mathbf{DM}(\udl{R}, \nabla) $.  
	\end{thm}
	\begin{proof} If $\Rbar$ a perfect $k$-algebra, this is a (unpublished) result of Gabber, relying on a result of Berthelot \cite{BerthelotDieudMod} where the case of a perfect discrete valuation ring is dealt; see also \cite{LauJAMS13} and \cite{ScholWeinModuli13} for different proofs. In this case, the connection can even be suppressed in the definition of a Dieudonn\'e module. If $\Rbar$ is a smooth $k$-algebra of finite type, this follows from  \cite[4.1.1, 2.3.4, 2.4.8]{deJong95}: indeed, since $\bar R$ satisfies \cite[1.3.1.1]{deJong95} by (1.3.2.1) in loc. cit.  and $\mathfrak{X}=\Spec \bar R$ satisfies the hyperthesis of \cite[4.1.1]{deJong95}  by (2.4.7.2) in loc. cit.
		
	\end{proof}
	
	\subsection{Classification of $p$-divisible groups over $R$}\label{BT/R}
	
	The same setting as in the previous subsection \S \ref{S:BT/Rbar}.
	Now we start with a $ p $-divisible group $ H $ over $ R $ and write
	$ \bar{H}=H\otimes_{R}\bar{R} $. For the  $p$-divisible group $H^*$  there is constructed in \cite[IV. 1.14]{MessingBT} a universal extension 
	$0\to \omega_{H}\to \mathrm{E}(H^*) \to H^*\to 0$; 
	taking $\underline{\mathrm{Lie}}$ (following notation in loc. cit.), we get an exact sequence of locally free $R$-modules
	\[
	0\to \omega_{H}\to \underline{\mathrm{Lie}}\big(\mathrm{E}(H^*)\big) \to \underline{\mathrm{Lie}}(H^*)\to 0
	\]
	where $\omega_{H}$ is the sheaf of invariant differential of $H$.
	Moreover, it follows from the construction of $\mathbb{D}^{*}(\bar{H})$ that we have canonical isomorphism 
	$\underline{\mathrm{Lie}}\big(\mathrm{E}(H^*)\big)\cong \mathbb{D}^{*}(\bar{H})(R)$ 
	of $R$-modules (\cite[IV. 2.5.4]{MessingBT}, see also  \cite[3.3.5]{BBM}) and thus we can identify them. 
	Similarly we have an exact sequence for $\bar H$,
	\[
	0\to \omega_{\bar H}\to \mathbb{D}^{*}(\bar{H})(\bar R) \to \underline{\mathrm{Lie}}(
	\bar{H}^*)\to 0
	\]
	Here we stress that the submodule
	$ \omega_{H}\sbt \mathbb{D}^{*}(\bar{H})(R)$ is a locally direct summand of $\mathbb{D}^{*}(\bar{H})(R)$ which lifts the locally direct summand $ \omega_{\bar H}\sbt \mathbb{D}^{*}(\bar{H})(\bar R)$ of $\mathbb{D}^{*}(\bar{H})(\bar R)$.
	Let $\underline{M}\in \big(\mathbf{BT}/\bar{R}\big)$ be the Dieudonn\'e module  corresponding to $\bar H$. Write $\bar \FF: \bar{M}^{\sig} \to \bar M$ for the reduction modulo $p$ of $\FF$. Set $\bar{M}^1:=\omega_{\bar H}$; it is called the \textbf{Hodge filtration} of $ \bar{M} $. Then we have the relation,	\begin{equation}\label{HodgeFil}		\bar{M}^{1,\sig}=\Ker (\bar \FF)\sbt \bar{M}^{\sig}.	\end{equation}
	
	Denote by 
	$\big(\mathbf{BT}/R\big)$ the common category of 
	$p$-divisible groups over $\Spec R$ and over $\Spf R$ (justified by \cite[2.4.4]{deJong95}) and by $ \mathbf{AFDM}(\underline{R}, \nabla) $ the category of tuples $(M, M^1, \FF, \VV, \nabla_{M}) $, 
	where
	$ (M,  \FF, \VV, \nabla_{M}) $
	is an object in $\mathbf{DM}(\underline{R}, \nabla)$, and where $M^1\sbt M$ is a locally direct summand, lifting the locally direct summand 
	${\bar M}^1\subset \bar M$.
	Morphisms are obvious ones.  We call an object in 
	$ \mathbf{AFDM}(\underline{R}, \nabla) $ 
	an \textbf{admissibly filtered Dieudonn\'e module over $R$} over $\underline{R}$ (or simply over $R$, when $\sig$ is chosen); cf. \cite[V. 1.4]{MessingBT}.

	\begin{rmk}\label{AbComp}
		For the purpose of future reference, we recall the following well-known comparison results, that underlines the crystalline Dieudonn\'e theory. If $\sA$ is an  abelian scheme over $R$, with $\bar \sA$ its pullback to $\bar R$, we have  canonical isomorphism of $R$-modules (\cite[V. 2.3.7]{Bert74}, also cf. \cite[7.26.3]{BO78})
		\begin{equation}\label{CrideRhamComp}
			\HH^1_{\dR}(\sA/R)\cancong \HH^1_{\cris}(\bar \sA/R).
		\end{equation}
		Moreover, we have the following a canonical isomorphism of filtered $R$-modules
		\begin{equation*}
			\big( \mathbb{D}^{*}(\bar{H})( R)\spt \omega_{H}\big) \cancong \big(\HH^1_{\dR}(\sA/R) \spt \omega_{A}\big).
		\end{equation*}
	\end{rmk}

	\begin{thm}\label{ClaBT}
		The assignment $G\mapsto \big(\mathbb{D}^{*}(\bar H)(R), \omega_{H}, \FF, \VV, \nabla_{M}\big)$ gives a category equivalence between $\big(\mathbf{BT}/R\big)$ and $ \mathbf{AFDM}(\underline{R}, \nabla) $.
	\end{thm}
	
	\begin{proof}
		To lift a $p$-divisible group $\bar H$ over $\bar R$ to $R$ is the same thing as lifting its dual ${\bar H}^*$ to $R$, the assertion follows from the combination of Thm. \ref{ClaBTmodp} and Grothendieck-Messing deformation theory (\cite[V, 1.6]{MessingBT}) which in our setting says that lifting ${\bar H}^*$ to $R$ is equivalent to lifting the locally direct summand $ \omega_{\bar H}\sbt \mathbb{D}^{*}(\bar{H})(\bar R)$ to a locally direct summand of $\mathbb{D}^{*}(\bar{H})(R)$. 
	\end{proof}

	\subsection{Base change along simple frames}\label{BCofFrame}
	
	Let $\bar R'$ be as in Exam. \ref{ExofSimpFram} and $\underline{R'}=(R', \sig')$ a simple frame of $\bar R'$ over $W(k)$. Let $f: \underline{R}\to \underline{R'}$ be a morphism of simple frames over $W(k)$ (i.e., a map $f: R\to R'$ of $W(k)$-algebras, compatible with Frobenius lifts). Then we have commutative diagrams as below induced by base change along $f$ in obvious senses.
	\begin{equation}\label{Functrltyfram}
		\xymatrix{\big(\mathbf{BT}/\bar R\big)\ar[r]^{\cong\ \ }\ar[d]&\mathbf{DM}(\underline{R}, \nabla)\ar[d]\\
			\big(\mathbf{BT}/\bar R'\big) \ar[r]^{\cong\ \ } &\mathbf{DM}(\underline{R'}, \nabla)}
		\ \ \ 
		\xymatrix{\big(\mathbf{BT}/R\big)\ar[r]^{\cong\ \ }\ar[d]&\mathbf{AFDM}(\underline{R}, \nabla)\ar[d]\\
			\big(\mathbf{BT}/R'\big) \ar[r]^{\cong\ \ } &\mathbf{AFDM}(\underline{R'}, \nabla)}
	\end{equation}
	
	\subsection{Partially divided Frobenius}\label{ParDivFrob}
	The setting is the same as in the previous two subsections \S \ref{S:BT/Rbar}, \S \ref{BT/R}. Let  $ (M, M^1, \FF, \VV, \nabla_{M}) $ be an object in $ \mathbf{AFDM}(\underline{R}, \nabla) $. 
	 Assume now that  the submodule $M^1\subset M$ is a (not just locally) direct summand of $M$.
	 Let $M=M^1\oplus M^0$ be a decomposition of $M$ into $R$-submodules; such a decomposition is called a \textbf{normal decomposition} of $\underline{M}$ (or simply of $ M $).  Define the following maps,
	\begin{equation}\label{defGamma}
		\Gamma:= \frac{1}{p}\FF|_{M^{1, \sig}}\oplus \FF|_{M^{0, \sig}},\ \ \  \rf:=p\id_{M^{1, \sig}}\oplus \id_{M^{0, \sig}},
	\end{equation}
	so that we have $\FF= \Gamma\circ\rf$. We shall call $\Gamma$ the \textbf{partially divided Frobenius} of $M$ w.r.t. the normal decomposition $M=M^1\oplus M^0$. The next lemma describes the most important property of $\Gamma$, with the point being that a normal decomposition of $M$ enables us to decompose $\FF$ as the composition of an integral part $\Gamma$ with a rational part $\rf$. Such a decomposition is important for later applications.   
	
	\begin{lem}\label{Zipisom}
		The map $\Gamma$ defiend in \eqref{defGamma} is an isomorphism of $R$-modules.
	\end{lem}
	
	\begin{proof}

		
		Let us first note that $\Gamma$ is surjective: indeed from \eqref{HodgeFil} we obtain the equality displayed below, which implies $\Im(\Gamma)=M$:
		\[
		\FF^{-1}(pM)=\pi^{-1}(\bar{M}^{1,\sig})=M^{1,\sig}+p M^{\sig}=M^{1,\sig}\oplus p M^{0,\sig},
		\]
		where $\pi: M^{\sig}\to \bar{M}^{\sig}$ is the canonical reduction modulo $p$ map.
		
		It is enough to show  that for every  maximal ideal $\mathfrak{m}$ of $R$, the pull back to $\widehat{R}_{\mathfrak{m}}$ of $\Gamma$ is an isomorphism. To ease notation, write $A=\widehat{R}_{\mathfrak{m}}$ and let $(A, \sig)$ be the unique simple frame of the $\mathfrak{m}$-adic completion  of $\bar R$,  induced by $(R, \sig)$ as in Lem. \ref{Froblifts}. By functoriality as discussed above \eqref{Functrltyfram}, the base change along $(R, \sig)\to (A, \sig)$ of $\underline{M}$ is equal to the admissibly filtered Dieudonn\'e module of $H\otimes_RA$ if $H$ is the $p$-divisible group over $R$ corresponding to $\underline{M}$. So we are reduced to show the $\Gamma$ map over $(A, \sig)$ corresponding to $H\otimes_RA$ and the decomposition $ M_A=M_A^1\oplus M_A^0$ is an isomorphism. But as $A$ is local, the source and target of $\Gamma:M^{\sig}\to M$ are free $A$-modules of the same rank, by Nakayama's lemma the assertion follows from the surjectivity of $\Gamma$. 
	\end{proof}

We remark that in later sections we will not use the full power of the classification results, Thm. \ref{ClaBTmodp} and Thm. \ref{ClaBT}. We only need the fact that given a $p$-divisible group over $\Rbar$ (resp. over $R$). one can associate with it an object in $\mathbf{DM}(\underline{R}, \nabla)$ (resp. in $ \mathbf{AFDM}(\underline{R}, \nabla) $) and such an association is  compatible with base change of simple frames.

	\section{Good reduction of Shimura varieties of Hodge type}
	\label{S:shivar}
	
	\subsection{Shimura varieties of Hodge type} 
	
	Let $\textbf{G}$ be a (connected) reductive group over $\mathbb{Q}$ and $\textbf{X}$ a $ \mathbf{G}(\mathbb{R})$ conjugacy class of homomorphisms
	\[
	h:\mathbb{S}:=\Res_{\mathbb{C}/\mathbb{R}}\mathbb{G}_m\to \textbf{G}_{\mathbb{R}}
	\]
	of algebraic groups over $\mathbb{R}$, such that $ (\mathbf{G}, \mathbf{X}) $ is a Shimura datum in the sense that they satisfy axioms (2.1.1.1)-(2.1.1.3) of \cite[2.1.1]{DeligneCanonicalmodel}.  Suppose that  $V$ is a finite-dimensional $\mathbb{Q}$-vector space with a perfect alternating pairing $\psi$ and write $\GSp=\GSp(V, \psi)$ for the corresponding group of symplectic similitudes. Then we get the most important example of Shimura datum $(\GSp, \textbf{S}^{\pm})$ with $\textbf{S}^{\pm}$ the Siegel double space, which is defined to be the set of homomorphisms $\mathbb{S}\to \GSp_{\mathbb{R}}$ such that: (1) The $\mathbb{C}^{\times}$ action on $V_{\mathbb{R}}$ gives rise to a Hodge structure of type $(-1, 0)$ and $(0, -1)$;
	(2) $(x, y)\mapsto \psi(x, h(i)y)$ is (positive or negative) definite on $V_{\mathbb{R}}$.
	
	In this paper we consider a  Shimura datum $(\textbf{G} , \textbf{X})$ of Hodge type; i.e.,  there exists an embedding of Shimura data  $ (\textbf{G} , \textbf{X})\hookrightarrow (\GSp, \textbf{S}^{\pm})$ for some $(\GSp, \textbf{S}^{\pm})$.  Let $ \mathsf{K}=\mathsf{K}_p\mathsf{K}^{p}\sbt \textbf{G}(\mathbb{A}_f) $ be an open compact subgroup such that $ \mathsf{K}_p \sbt  \mathbf{G}(\mathbb{Q}_p)$ is a hyperspecial subgroup and that $ \mathsf{K}^{p} \sbt \mathbf{G}(\mathbb{A}_{f}^{p})$ is sufficiently small (hence is neat). The condition that $ \mathsf{K}_p $ is hyperspecial means that there is a reductive group $ \mathcal{G} $ over $\mathbb{Z}_{(p)} $, which we fix from now on, such that $ \mathsf{K}_p=\mathcal{G}(\mathbb{Z}_p) $.
	The condition that $ \mathsf{K}^{p}$ is sufficiently small guarantees that the double quotient 
	\begin{equation*}
		\Sh_\mathsf{K}(\textbf{G}, \textbf{X})_{\mathbb{C}}:=\textbf{G}(\mathbb{Q})\backslash \textbf{X}\times \textbf{G}(\mathbb{A}_f)/\mathsf{K}
	\end{equation*}
	has the structure of a \textbf{smooth} quasi-projective complex variety by a theorem of Baily-Borel.  Results of Shimura, Deligne, Milne and others imply that, up to isomorphism, $\Sh_\mathsf{K}(\textbf{G},\textbf{ X})_{\mathbb{C}}$ has a unique quasi-projective smooth model $\Sh_{\mathsf{K}}(\textbf{G}, \textbf{X})$  over the reflex field $ E $ of $ (\textbf{G}, \textbf{X}) $.  The reflex field $ E $ only depends on the Shimura datum $ (\textbf{G},\textbf{ X}) $. For $ (\GSp, \textbf{S}^{\pm}) $, the reflex field is  $ \mathbb{Q} $. 
	
	\subsection{Integral canonical models}
	
	As explained in \cite[2.3.1, 2.3.2]{KisinIntegralModels}, for a given Shimura datum $(\textbf{G},\textbf{X})$  with embedding $(\mathbf{G}, \mathbf{X})\hookrightarrow  (\GSp, \textbf{S}^{\pm})$, using Zarhin's trick we may modify  $ (V, \psi ) $ so that there exists a $ \mathbb{Z}_{(p)} $-lattice $ \Lmd $ of $ V $ with the following property: (1). the pairing $ \psi $ induces a perfect $ \mathbb{Z}_{(p)} $-pairing on $ \Lmd $, still denoted by $ \psi $; (2) the embedding $ \mathbf{G} \to \GSp$ is induced by an embedding $  \mathcal{G} \hookrightarrow \GSp(\Lmd, \psi)$ of reductive group schemes over $ \mathbb{Z}_{(p)} $. From now on, we fix such an embedding and accordingly the modified embedding of Shimura data $(\mathbf{G}, \mathbf{X})\hookrightarrow  (\GSp, \textbf{S}^{\pm})$.  Set $ \tilde{\mathsf{K}}_{p}=\GSp(\mathbb{Z}_p) $. By \cite[2.1.2]{KisinIntegralModels} there exists an open compact subgroup $ \tilde{\mathsf{K}}^{p}\sbt \GSp(\mathbb{A}_f) $ containing $\mathsf{K}^{p}  $ such that $ \iota $ induces an embedding of Shimura varieties over $ E $,
	\[ \shk\hookrightarrow \Sh_{\tilde \K}\otimes_{\mathbb{Q}}E.\] Moreover, if $\tilde{\mathsf{K}}^{p} $ is sufficiently small, $ \Sh_{\tilde{\mathsf{K}}}$ has a quasi-projective smooth model over $ \mathbb{Z}_{(p)} $, denoted by $ \tilde{\sS} = \tilde{\sS}_{\K}$, which has an explicit moduli interpretation as described in (\cite[1.3.4]{KisinModpPoints}). In what follows we always assume that $\mathsf{K}^{p}$ and $\tilde{\mathsf{K}}^{p} $ are sufficiently small, and we will also fix a $\Z$-lattice $\Lmd_{\Z}$ of the $\Z_{(p)}$-module $\Lmd$ such that $\Lmd_{\Z}\otimes \Zhat$ is $\Ktil$-stable. The choice of such a $\Z$-lattice allows one to describe the scheme $\tilde{\sS}$ as moduli space of polarized abelian varieties (not just up to prime to $p$-isogeny). In particular, it comes with a universal abelian scheme, denoted by $ \sA$. 
	
	Fix a place $ v $ of $ E $ above $ p $. Denote by $ \mathcal{O}_{E, (v)} $ the localization at $ v $ of the ring of integers  $ \mathcal{O}_E $  of $ E $.  Denote by $ \sS =\sk(\textbf{G}, \textbf{X}) $ the normalization of the schematic closure of $ \shk$ in $ \tilde{\sS}\otimes_{\Z_{(p)}}\sO_{E, (v)} $. Recall that we have the assumption $p\geq 3$. The following theorem is now well-known and is due to Vasiu and Kisin, independently.  
	\begin{thm}
	The scheme $\sS$ is smooth over $\sO_{E,(v)}$ and is the  integral canonical model over $\sO_{E, (v)}$ of $\Sh_{\K}$. 
	\end{thm}
   Strictly speaking, \textbf{integral canonical model} refers to a tower of models $\{\sS_{\K}\}_{\K^p}$ over $\sO_{E, (v)} $ for the tower $\{\shk\}_{\K^p}$, with $\K=\K^p\K^p$ and $\K^p$ varying; see \cite[\S 2]{Milnepoinsmodgoodp} for its precise meaning. Here we abuse language since soon $\K$ will be fixed till the end of  this paper. 
    
    In particular, we obtain a finite morphism 	$\varepsilon: \mathcal{S}\to \tilde{\mathcal{S}}$  of schemes over $ \mathcal{O}_{E, (v)} $. We call the pull-back to $ \mathcal{S} $ of $ \mathcal{A}$ the \textbf{universal abelian scheme} of $ \mathcal{S} $, still denoted by $ \mathcal{A} $. Write $ \kap$ for the residue field of $ \mathcal{O}_{E, (v)} $ and $ S=S_{\K}$ for the special fibre of $ \sS$. In particular, $ S $ is a quasi-projective smooth scheme over $ \kap$, coming with a universal abelian scheme $ \sA=\sA_{\kap}$.  
	
	In fact, the existence of the hyperspecial subgroup $ \mathsf{K}_{p} $ implies that $ E $ is unramified at $ p $ (\cite[~4.7]{ShimuraVarietiesandMotives}), and hence we have 
	$\sO_{E, v}=W(\kap)$, where $ \sO_{E, v} $ is the completion of $ \sO_{E,(v)} $ with respect to its maximal ideal. In what follows, we will mainly work over $W(\kap)$ or over $\kap$. We will use the same notations for the base change to $W(\kap)$ of those objects  defined over $O_{E,(v)}$ (e.g., the integral model ~$\sS$).

	\subsection{Reduction of Hodge cocharacters and their Frobenius twists}\label{S:cochar}
	
	As shown in \cite[1.3.2]{KisinIntegralModels}, the $ \mathbb{Z}_{(p)} $-reductive group scheme $ \mathcal{G} $ can be realized as the schematic stabilizer of a finite set of tensors $ (s_{\alpha})_{\alpha}\sbt \Lmd^{\otimes}=(\Lmd^*)^{\otimes} $; i.e., for any $ \mathbb{Z}_{(p)} $-algebra $ R$, 	\begin{equation*}	\mathcal{G}(R)=\{g\in \GL(\Lambda_{R}^*)\ \big| \ g(s_{\alpha, R})=s_{\alpha, R}, \forall 	\alpha\}, \end{equation*} where $s_{\alpha, R}\in (\Lams_R)^{\otimes}$ denotes the tensor induced by $s_{\alpha}$.	Here for the functoriality consideration later, we view  $\sG $ as a reductive $\Z_{(p)}$-subgroup scheme of  $ \GL(\Lmd^*)$ via the dual representation $\GL(\Lmd)\cancong \GL(\Lmd^*)$, \begin{equation}\label{EmbedGrp/Z(p)}
	\iota: \sG\hookrightarrow \GSp(\Lmd, \psi)\hookrightarrow\GL(\Lmd)\cancong \GL(\Lmd^*).
	\end{equation}
	 Write $ G $ for the special fibre of $ \mathcal{G} $. It is a (connected) reductive group over $ \Fp$. 
	
	For any $h\in \textbf{X}$, there is an associated Hodge cocharacter 
	$\nu_h: \mathbb{G}_{m,\mathbb{C}}\to \textbf{G}_{\mathbb{C}}$
	which can be described as follows. For any $\mathbb{C}$-algebra $R$, we have $R\otimes_{\mathbb{R}}\mathbb{C}=R\times c^*(R)$
	where $c$ denotes complex conjugation.  Then on $R$-points $\nu_h$ is given by
	\[
	R^{\times}\hra R^{\times}\times c^*(R)^{\times}= (R\otimes_{\mathbb{R}}\mathbb{C})^{\times}=\mathbb{S}(R)\xrightarrow{h}\textbf{G}_{\mathbb{C}}(R),
	\]
	where the first inclusion is given by $ x\in R^{\times}\mapsto (x,1) $. Denote by $[\upmu]_{\mathbb{C}}$ the unique $\textbf{G}(\mathbb{C})$-conjugacy class in $\Hom_{\mathbb{C}}(\mathbb{G}_{m, \mathbb{C}},\textbf{G}_{\mathbb{C}})$ 
	which contains \textbf{the inverses} of all the $\nu_h$'s. Let $\mathsf{Z} = \textbf{Hom}_{\mathbb{Z}_{(p)}} (\mathbb{G}_{m,\mathbb{Z}_{(p)} }, \mathcal{G})$ be the fppf sheaf of cocharacters, and $ \mathsf{Ch} =\mathcal{G}\backslash\mathsf{Z} $ the fpqc quotient sheaf of $ \mathsf{Z} $ by the adjoint action of $ \mathcal{G} $. By \cite[Chap. XI, Cor. 4.2]{SGA3}, the sheaf $ \mathsf{Z}$ is represented by a smooth separated scheme over $ \mathbb{Z}_{(p)} $, and it is shown in \cite[2.2.2]{ChaoZhangEOStratification} that $ \mathsf{Ch} $  is represented by a disjoint union of connected finite  \'etale  schemes over $ \mathbb{Z}_{(p)} $. Moreover, it is shown in loc. cit. that the  $\mathbb{C}$-point of $ \mathsf{Ch} $ corresponding to the conjugacy class $[\upmu]_{\mathbb{C}}$ descends to a $W(\kap)$-point $ \mathsf{Ch}$. We call the resulting $\kap$-point of $\Ch$ \textbf{the reduction over $\kap$ of $[\upmu]_{\C}$} and denote it by $[\upmu]_{\kap}$. In fact the conjugacy class $[\upmu]_{\kap}$ admits a representative \begin{equation}\label{Def:mu}	\upmu:\mathbb{G}_{m,\kap}\to G_{\kap}. 	\end{equation}  We choose such a representative $ \upmu$ but note that it is the $ G(\kap) $-conjugacy class $[\upmu]_{\kap}$ that is canonically determined by the Shimura datum $(\mathbf{G}, \mathbf{X})$. We define a Frobenius twist of $\upmu$, 	\begin{equation}\label{Def:musig}	\sigmu :=\sig(\upmu ) : \mathbb{G}_{m, \kap}= \mathbb{G}_{m, \kap}^{\sig}\xrightarrow{\sig^*\upmu} G_{\kap}^{\sig}= G_{\kap},\end{equation}  															where  $\sig^*\upmu: \mathbb{G}_{m, \kap}^{\sig}\to G_{\kap}^{\sig}$ is the base change of $ \upmu $ along the absolute Frobenius  $\sig: \kap \to \kap$ of $\kap$. Here we suppress notations $\cancong$ and identify $ G_{\kap}^{\sig} $, resp. $ \mathbb{G}_{m, \kap} ^{\sig}$ with $G_{\kap}$, resp. $\mathbb{G}_{m,\kap}$ since the latter ones are defined over $ \mathbb{F}_p $. 
	
	Every element $h\in \textbf{X}$ defines a Hodge decomposition $V_{\mathbb{C}}=V^{(-1, 0)}\oplus V^{(0,-1)}$ via the embedding $\textbf{X}\hookrightarrow \textbf{S}^{\pm}$. By definition of $  \textbf{S}^{\pm} $,  $\nu_{h}(z)$ acts on $V^{(-1, 0)}$ through multiplication  by $z$ and on $V^{(0, -1)}$ as the identity. In particular, $ \nu_{h}$ is of weight $ 1$ and $ 0 $, and hence $ \upmu: \Gmka\to \Gka $  is of weight $ -1 $ and $ 0 $.
	%
	Since the scheme  $ \mathsf{Z} $ is smooth, there exists a  $\sG_{\wkap}$-valued cocharacter   $\tilde{\upmu}: \mathbb{G}_{m, W(\kap)} \to \mathcal{G}_{W(\kap)}$, which lifts $ \upmu $.   From now on we fix such a lift $ \mutil $, and define a Frobenius twist of $\mutil$  along the Frobenius  $ \sig: W(\kap) \to W(\kap)$ as follows
	\begin{equation} \label{Def:tilmu}
		\tilde{\upmu}^{\sig}:=\sig(\tilde{\upmu}):  \mathbb{G}_{m, W(\kap)}= \mathbb{G}_{m, W(\kap)}^{\sig}\xrightarrow{\sig^* \mutil}  \mathcal{G}_{W(\kap)}^{\sig}= \mathcal{G}_{W(\kap)},
	\end{equation}
	where we identify
	$ \mathcal{G}_{W(\kap)}^{\sig} $,  resp. $ \mathbb{G}_{m, W(\kap)}^{\sig}$ with
	$ \mathcal{G}_{W(\kap)} $, resp. $ \mathbb{G}_{m, W(\kap)}$ 
	since they are  defined over $ \mathbb{Z}_p $ already. Clearly $ \sigmu=\sig(\upmu) $ is the reduction modulo $ p $ of $\tilde{\upmu}^{\sig}=\sigmutil $. 
	The cocharacter $ \tilde{\upmu}$ induces  weight decomposition of $\Lambda_{W(\kap)}$ and $\Lmd^*_{W(\kap)}$
	\begin{equation}\label{Weighdecomp}
		\Lambda_{W(\kap)}=\Lambda_{W(\kap)}^0\oplus\Lambda_{W(\kap)}^{-1}, \ \ \ \ 
		\Lmd^*_{W(\kap)}=\Lmd^{*,0}_{W(\kap)}\oplus\Lmd^{*,1}_{W(\kap)}.
	\end{equation}  
\subsection{Some group-theoretic preparations}\label{S:Emdcoord} We first introduce some subgroup schemes of $\sG_{\wkap}$ that are induced by $\mutil$. 
	Denote by 
	$\sP_{+}=\sP_{\upmu}\sbt \sG_{\wkap} $ the scheme theoretic stabilizer of the filtration $\Lmd_{\wkap}\spt \Lmd_{\wkap}^{-1}$ (equivalently, of the filtration $ \Lmd^{*}_{\wkap}\spt \Lmd^{*,1}_{\wkap}  $ via dual representations). It is a parabolic subgroup scheme of $ \sG_{\wkap} $. Similarly we denote by $ \sP_{-}=\sP_{\upmu^{-1}}\sbt \sG_{\wkap} $ the opposite subgroup scheme of $ \sP_{+} $. Write  $\sU_{\pm}=\sU_{\pm}(\upmu)\sbt \sP_{\pm}$ for the  corresponding unipotent radicals, and $\sM=\sP_+\cap \sP_-$ for the common Levi subgroup scheme of $\sP_-$ and $\sP_+$. Note that $\sM$ is also the centralizer in $\sG_{\wkap}$ of  $\mutil$. 
	
	The next lemma will  become useful in later sections; it can also be seen easily using the embedding $\sG_{W(\kap)}\hookrightarrow \GL_{2\rg, W(\kap)}$ to be discussed right after this lemma.
	
	\begin{lem}\label{Lem:integral}
		Let $ A $ be a flat $W(\kap)$-algebra such that $ \bar A:=A/pA\neq 0 $. Then we have  
		\[
		\mutil(p) \sP_{+}(A)\mutil(p)^{-1}\sbt \sG(A), \ \ \  \mutil(p) \sU_{+}(A)\mutil(p)^{-1}\sbt \mathtt{K}_1(\sG)(A),
		\]
		with $\mathtt{K}_1 (\sG)(A):= \{g\in \sG(A)| \overline{g}=1\in G(\bar A)\}$,
		where $ \bar g$ denotes  the image of $g$ in $G(\bar A)$ under the canonical reduction map  $\sG(A)\to G(\bar A)$.
	\end{lem}
	
	\begin{proof}
		Recall the dynamic descriptions of $ \sP_+ $ and $ \sU_+ $ (see for example \cite[2.1]{PseudoReductiveGroups}):
		\begin{align}
			\begin{array}{l}
				\sP_{+}(A)=\{g\in \sG(A)\ \ | \  \ \lim\limits_{t\to 0}\mutil(t)g\mutil(t)^{-1}\  \text{exists}\},\\
				\sU_{+}(A)=\{g\in \sP^+(A)\ | \  \ \lim\limits_{t\to 0}\mutil(t)g\mutil(t)^{-1}=1 \},
			\end{array}
		\end{align}
		where the condition  $\lim\limits_{t\to 0}\mutil(t)g\mutil(t)^{-1}$ exists means that the homomorphism of $ A $-group schemes $ f_{\mutil, g}: \mathbb{G}_{m,A}\to \sG_{A}, \ t\mapsto  \mutil(t)g\mutil(t) ^{-1},$ extends to a morphism of $ A $- schemes $F_{\mutil, g}: \mathbb{G}_{a, A}\to \sG_{A}  $, while the condition $\lim\limits_{t\to 0}\mutil(t)g\mutil(t)^{-1}=1$ requires  $ F_{\mutil, g}(0) =1\in \sG(A)$.

		Now let $ g\in \sP_+(A)  $. Since $ p\in \mathbb{G}_{m}(A[\frac{1}{p}]) \cap \mathbb{G}_{a}(A)$ ($A$ is $p$-torsion free),  one finds that 
		\[
		\mutil(p)g\mutil(p)^{-1}=f_{g, \mutil}(p)=F_{\mutil, g}(p)\in G(A).
		\]
		If moreover, $ g\in\sU_+(A), $ the functoriality of $F_{\mutil, g}$ for the canonical projection $A\to \bar A$, viewed as a map between $W(\kap)$-algebras, implies
		\[
		\overline{\mutil(p)g\mutil(p)^{-1} }= \overline{F_{\mutil, g}(p)}=1.
		\]
		
	\end{proof}

    For later applications, we fix an embedding of $\sG_{W(\kap)}$ into $\GL_{2\rg, W(\kap)}$ as follows. Choose a $ W(\kap) $-basis 
	\[
	v_1, \cdots, v_{\rg},  v_{\rg+1}, \cdots, v_{2\rg} \in \Lmd^{*}_{\wkap}
	\]
	such that the first $\rg$ elements above lies in $ \Lmd^{*,1}_{\wkap} $ and the remaining ones lie in $ \Lmd_{\wkap}^{*, 0} $. Then by sending an element $ h \in \GL(\Lams_{\wkap})$ to the matrix $ X_h \in \GL_{2\rg, \wkap}$ such that 							\[ h(v_1, \cdots, v_{2\rg})=(v_1, \cdots, v_{2\rg})X_h, \]
	 we obtain an isomorphism of $ \wkap $-group schemes between $\GL(\Lmd^*_{\wkap})$ and $\GL_{2\rg, \wkap}$. Hence from \eqref{EmbedGrp/Z(p)} we obtain an embedding of $\sG_{W(\kap)}$ into $\GL_{2\rg, W(\kap)}$, as $ \wkap $-reductive group schemes,
	\begin{align}\label{keyembedding}
		\iota: \sG_{\wkap} \hra \GSp(\Lmd, \psi)_{\wkap}\hra\GL(\Lmd^*_{\wkap})\cong\GL_{2\rg, \wkap},
	\end{align}
	and accordingly a cocharacter  $\mutil':=\iota \circ \mutil$ of $ \GL_{2\rg, \wkap}$.
	For every  $ \wkap $- algebra $ R$ such that  $p\in R^{\times}$, we have
	\begin{equation}
		\mutil'(p)=\big(\begin{array}{cc}
			p\rm{I}_{\rg}&\\
			& \rm{I}_{\rg}
		\end{array}
		\big)\in \GL_{2\rg}(R).   
	\end{equation}
      We denote by $ \sP'_{\pm},  \sU'_{\pm},  \sM' $ the counterparts of  $\sP_{\pm}, \sU_{\pm}, \sM$ respectively for the cocharacter $ \tilde{\upmu}'$ of $\GL_{2\rg, \wkap} $. Clearly these subgroups can be described explicitly in term of matrices; for example  $ \sU'_-\sbt  \sG'_{W(\kap)}  $ consists of matrices of the form $\Big(\begin{array}{cc}	\rm{I}_{\rg}&\\	*& \rm{I}_{\rg}	\end{array}\Big),$ where $ * $ denotes a $ \rg$ by $\rg $-matrix.  It is a general fact that we have (see \cite[4.1.10]{ConradReductiveGroupSchemes} for example)	\begin{equation}\label{Intersect}	\sP_{\pm}=\sP'_{\pm}\cap \sG,\ \ \  \sU_{\pm}=\sU'_{\pm}\cap \sG,  \ \ \  \sM=\sM'\cap \sG. \end{equation}
	
	We shall see that  the embedding $ \iota $ in (\ref{keyembedding}) will enable us to reduce some group-theoretic arguments in later sections to much easier problems like multiplying $2$ by $2$ block matrices.

	  \subsection{Tensors on $\HH_{\dR}^1(\sA/\sS)$}\label{S:tensors}
	For all $i\geq 0$, write $\HH^i_{\dR}(\mathcal{A}/\mathcal{S}):=\textbf{R}^i\pi_* (\Omega^{\bullet}_{\sA/\sS})$ for the $i$-th relative de Rham cohomology of $\sA$ over $\sS$, where $\pi: \sA\to \sS$ is the structure morphism. As is shown in \cite[2.5.2]{BBM} (generalizing the well-known case where the base is a field to the case where the base is an arbitrary scheme), for all $i\geq 0$ (resp. all $r, s\geq 0$), the $\sO_{\sS}$-module $\HH^i_{\dR}(\mathcal{A}/\mathcal{S})$ (resp. 
	$ \textbf{R}^s\pi_* (\Omega^{r}_{\sA/\sS})$) are locally free and their formations commute with arbitrary base change. Moreover, the Hodge-de Rham spectral sequence 
	\[
	{}_{\rm H}\mathrm{E}^{r,s}=\textbf{R}^s\pi_* (\Omega^{r}_{\sA/\sS}) \implies \HH^{r+s}_{\dR}(\sA/\sS),  
	\]
	 degenerates at $\rm E_1$-page. In particular, we have an exact sequence of locally free $\sO_{\sS}$-modules 
	\begin{equation}\label{GenerHodgFil}
		0\to \omega_{\sA/\sS}\to  \HH^1_{\dR}(\mathcal{A}/\mathcal{S})\to \textbf{R}^1\pi_* \sO_{\sS}\to 0,  
	\end{equation}
	where the Hodge filtration $\omega_{\sA/\sS}= \pi_*\Omega^1_{\sA/\sS}$ is of rank $\rg$ and  $\HH^1_{\dR}(\mathcal{A}/\mathcal{S})$ is of rank $2\rg$. 
	
        For typographical reason, in this and the next subsections we write $\sV_{\dR}$ for $\HH^1_{\dR}(\mathcal{A}/\mathcal{S})$. 
	    Below we explain the so-called ``(integral) de Rham tensors" on $\sV_{\dR}.$ We will need these tensors to define interesting torsors over $\sS$ in \S \ref{Torsors}. 
	    
	     The $\q$-representation $V$ of $\mathbf{G}$ coming from the embedding $\mathbf{G}\hookrightarrow \GSp(V, \psi)$ gives rise to a $\q$-local system $\sV_{B, \q}=\RI\pi^{\an}_* \q$ on $\shkcan$. Below we  first explain how the tensors $(s_{\alpha})_{\alpha}\sbt V^{\otimes}$ that cut out $\mathbf{G}$ inside $\GL(V^*)$ induce global sections on $\sV_B^{\otimes}$;  cf. \cite[2.2]{KisinIntegralModels} and \cite[2.3]{CarainiScholzeGenericCompact}.  Write: \[\tilde{\shk}=X\times \mathbf{G}(\Afp)/\K,\ \ \  \tilde{\Sh_{\K'}}=S^{\pm}\times \GSp(\Afp)/\K'.\] Then the canonical projection $\tilde{\Sh_{\K}}\to \Sh_{\K, \C}^{\an}$ (resp. $\tilde{\Sh_{\K'}}\to \Sh_{\K', \C,\an}$) makes $\tilde{\shk}$ a $\mathbf{G}(\q)$-torsor over $\shkcan$ (resp. $\tilde{\Sh_{\K'}}$ a $\GSp(\q)$-torsor over $\Sh_{\K', \C}^{\an}$). To make distinctions, we write $\sA'$ for the universal (analytic) abelian variety over $\Sh_{\K', \C}^{\an}$ with $\pi': \sA'\to \Sh_{\K', \C}^{\an}$ the structure map. Then we know that the isogeny class of $\sA'$  corresponds to the dual of the $\q$-local system $\RI\pi_*'\q$ (viewed as a variation of Hodge structure over  $\Sh_{\K', \C}^{\an}$), which in turn corresponds to the constant $\q$-local system $V$ over the cover $\tilde{\Sh_{\K'}}$, together with the structure morphism $\GSp(\q)\to \GL(V)$. Clearly we have the following commutative diagram \[\xymatrix{\tilde{\shk}\ar[d]\ar[r]&\tilde{\Sh_{\K'}}\ar[d]\\	\shkcan\ar[r]&\Sh_{\K', \C}^{\an},}\] where the top horizontal map is equivariant w.r.t. the group homomorphism $\mathbf{G}(\q)\to \GSp(\q)$. Hence the variation of Hodge structure $\sV_{B,\q}$ over $\shkcan$, that corresponds to the isogeny class of $\sA$, also corresponds to the constant $\q$-local system $V^*$ over the cover $\tilde{\shk}$, together with the representation $\mathbf{G}(\q)\hookrightarrow \GSp(\q)\hookrightarrow \GL(V)\cancong \GL(V^*)$. Now it is clear that the set of tensors $(s_{\alpha})_{\alpha}\sbt V^{\otimes}$ gives rise to a set of global sections (we simply call them \textbf{Betti-tensors}),	\begin{equation}\label{BetTensor}	(s_{\alpha, B})_{\alpha}\sbt \Gamma\big(\shkcan, \sV_{B, \q}^{\otimes}\big).
		\end{equation}
	   
	    By the Riemann-Hilbert correspondence \cite{DeligneDiffEq} we have the following equivalences of tensor categories, \[\mathrm{Loc}_{\C}(\shkcan)\xrightarrow[\cong]{(\cdot)\otimes\sO_{\shkcan}}\mathrm{VBIC}(\shkcan)\xleftarrow[\cong]{(\cdot)^{\an}} \mathrm{VBIC}(\shkc)^{\reg},\] 
	where $\mathrm{Loc}_{\C}(\shkcan)$ denotes the tensor category of $\C$-local systems over $\shkan$, $\mathrm{VBIC}(\shkcan)$ (resp. $ \mathrm{VBIC}(\shkc)^{\reg}$) denotes the tensor category of holomorphic (resp. algebraic) vector bundles with integrable connections (resp. with integrable connections, with regular singularities at infinity). Under these category equivalences, the $\C$-local system $\sV_{B, \C}:=\sV_{B, \q}\otimes \C$ corresponds to the vector bundle $\sV_{\dR, \C}:=\HH^1_{\dR}(\mathcal{A}/\shkc)$ over $\shkc$ and we have a parallel isomorphism of analytic vector bundles over $\shkan$, \[\upepsilon:\sV_{B, \C}\otimes\sO_{\shkcan}\cong \sV_{\dR,\C}^{\an},\]	where the left hand side is equipped with trivial connections. Hence by transport of structure, we obtain from the Betti tensors \eqref{BetTensor} our desired horizontal global sections (call them \textbf{de Rham tensors}), \begin{equation}\label{deRhamtensor/C} (s_{\alpha, \dR})_{\alpha}\sbt \Gamma\big(\shkc, \sV_{\dR,\C}^{\otimes}\big),
	\end{equation} 
such that $\upepsilon(s_{\alpha, B})=s_{\alpha, \dR}^{\an},$ with $s_{\alpha, \dR}^{\an}\in \Gamma(\shkcan, (\sV_{\dR,\C}^{\an})^{\otimes})$ understood. Here note that although $\sV_{B, \C}^{\otimes}$ does not live inside $\mathrm{Loc}_{\C}(\shkcan)$, each Betti tensor $s_{\alpha, B}$ lies in some direct summand of $\sV_{B, \C}^{\otimes}$ which does live inside $\mathrm{Loc}_{\C}(\shkcan)$. 

	\begin{prop}[{\cite[2.2.1, 2.3.9]{KisinIntegralModels}}]\label{dRTensorGlob}	Each of the de Rham tensors $s_{\alpha, \dR}$ in \eqref{deRhamtensor/C} descends to $\sO_{E, (v)}$; i.e., there exist (necessarily unique) horizontal global sections	\[(s_{\alpha, \dR})_{\alpha}\sbt \Gamma\big(\sS, \sV_{\dR}^{\otimes}\big),\] 	whose restriction on $\shkc$ are the tensors in \eqref{deRhamtensor/C}.	          \end{prop}	We call the global sections $s_{\alpha, \dR}$ obtained here \textbf{(integral) de Rham tensors}. 
	
	\subsection{Crystalline nature of integral de Rham tensors}\label{CrynatrofdRTensor}
	In this subsection we make clear of a (perhaps well-known) consequence of  Prop. \ref{dRTensorGlob} concerning the property of  integral tensors $s_{\alpha, \dR}$ being horizontal, as it will be needed later.
	
	 Let $R$ be a $p$-complete flat $\wkap$-algebra and $x, y: \Spec R\to \sS$ morphisms of $\wkap$-schemes which are congruent modulo $p$. Write $ \sV_{\dR, x}, s_{\alpha, \dR, x}$ the pull backs to $R$ along $x$ of $\sV_{\dR}, s_{\alpha ,\dR}$ respectively; similarly for $\sV_{\dR, y}, s_{\alpha, \dR, y}$.  It is a well-known fact that the vector bundle $\sV_{\dR}$ on $\sS$ has an $F$-crystal structure in the sense of \cite{KatzTravauxDwork} and the Gauss-Manin connection on $\sV_{\dR}$ provides a canonical isomorphism of $R$-modules (see for example (1.2) of loc. cit.), $\epsilon(x,y):  \sV_{ \dR, x}\cong\sV_{ \dR, y} $. Since $s_{\alpha, \dR}$ is horizontal,  we have \[\epsilon(x,y)\big(s_{\alpha, \dR,x}\big)=s_{\alpha, \dR,y}.\]
     \subsection{Torsors over  Shimura varieties}\label{Torsors}
     For simplicity, from now on we write $s$ instead of $(s_{\alpha})_{\alpha}$; similarly we simply write $s_{\dR}$ instead of $(s_{\alpha, \dR})_{\alpha}$. For a $\wkap$-morphism $x: \Spec R\to \sS$, we also write  $s_{\dR, R}$ for $s_{\dR, x}$ (i.e., the pull-back of $s_{\dR}$ along $x$), if the structure morphism $x$ is understood.   However, in order to keep notations suggestive, we still write $\HH_{\dR}^1(\sA/\sS)$ instead of $\sV_{\dR}$.  Now we are ready to define two $\sS$-schemes below, which will play important roles later,
	\begin{align*}
		&\mathbb{I}:=\Isom_{\sO_{\mathcal{S}}}\big([\Lmd^*_{W(\kap)}, s_{W(\kap)}]\otimes \sO_{\mathcal{S}},\  \  [ \HH^1_{\dR}(\mathcal{A}/\mathcal{S}), s_{\dR}]\big),\\
		&\mathbb{I}_+:=\Isom_{\sO_{\mathcal{S}}}\big([\Lmd^*_{W(\kap)}\spt \Lmd^{*,1}_{W(\kap)}, s_{W(\kap)} ]\otimes \sO_{\mathcal{S}}, \ \ [\HH^1_{\dR}(\mathcal{A}/\mathcal{S})\spt \omega_{\sA/\sS},  s_{\dR}]\big).
	\end{align*}
    Unwinding definition: for every $W(\kap)$-algebra $R$, a point $\xflat\in \II_+(R)$ consists of a pair $(x, \betx)$, where $x\in \sS(R)$ corresponds to a morphism $x: \Spec R\to \sS$ of $W(\kap)$-schemes, and 
	\[
	\betx: \big(\Lams_R\spt \LamsI_R\big)\cong \big(\HH^1_{\dR}(\sA_x/R)\spt \omega_{x}\big)
	\]
	is an isomorphism of $R$-modules, which maps $s_R$ to $s_{\dR, R}$ termwise. Here following our notational convention we denote $\sA_x, \omega_x, s_{\dR, R}$ the pull back to $R$ along $x$ of $\sA$,  $\omega_{\sA/\sS}, s_{\dR}$ respectively.
	We have similar descriptions for points $x\in \II(R)$ by omitting filtrations in $\betx$. 
	
	Clearly $\sG$ resp. $\sP_+$ naturally acts on $\II$ resp. $\II_+$ on the right, freely and transitively.  To be precise, the action of a section $ h\in \sG(R) $ (resp.  $h\in \sP_+(R) $) on $ \II(R) $ (resp. on $\II_+(R)$) is given by 
	\[
	\xflat\cdot h=(x,\betx)\cdot h= (x, \betx h).
	\]
	
	\begin{lem} \label{Torsor/IntS}
		The scheme $\II_+$ (resp. $ \II $) is a $\sP_+$-torsor (resp. $\sG$-torsor) over $\sS$. 
	\end{lem}
	\begin{proof} We only show here the assertion for $\II_+$ as the assertion for $\II$ can be shown in the same way; or maybe better, it follows from the fact that $\II$ is the push-forward of $\II_+$ along the homomorphism $\sP_+\hookrightarrow \sG$. 
		 
	Since $\II_+$ is an $\sS$-scheme of finite presentation and the action of $\sP_+$ on \IIplus\  is free and transitive,  it suffices to show that $\IIplus$ is faithfully flat over~$\sS$. In other words, we need to show that for each closed point $ s $ of $ \sS $, the pullback of $\II_+$ to $\Spec \sO_{\sS, s}$ along the natural map $\Spec \hat{\sO}_{\sS, s}\to \II_+$, denoted by $\II_{+,\hat{s}}$, is a $\sP_+$-torsor over $ \Spec  \hat{\sO}_{\sS, s}$.  Since we know already that when restricted to the generic fibre $\shk$ of $\sS$, $\II_+$ is a $\sP_+$-torsor over $\sS$ (as it is so after further base change to $\shkc$), we may assume that $s$ lies in the special fibre of $\sS$. But as stated in \cite[Lem. 2.3.2,  2)]{ChaoZhangEOStratification}, it is essentially shown in \cite{KisinIntegralModels} that $ \II_{ +,\hat{s}}$ is a trivial $ \sP_+ $-torsor over $ \widehat{\mathcal{O}}_{\mathcal{S}, s} $. 
		
	\end{proof}	

	\section{\protect{The zip period map $ \zet $ for $S$}}
	 In this section we review the zip period map $	\upzeta: S\longrightarrow \Gzips$, that Zhang constructs in \cite{ChaoZhangEOStratification}, following loc. cit. and \cite{Wortmann}. 

	\subsection{The stack of $G$-zips} \label{S: G-zips}
	Let $\upmu:\Gmka\to \Gka$ be as in \eqref{Def:mu}, and \[M, \ U_{\pm}\sbt P_{\pm},\] 		the special fibres of the algebraic groups $\sM, \sU_{\pm}\sbt \sP_{\pm}$  defined in \S \ref{S:Emdcoord}.  Recall our notational  convention in \S \ref{NotaConven}: for a subgroup $H\sbt \Gka$, we write $H^{\sig}\sbt \Gkasig\cancong \Gka$ for its base change along $\sig: \kap\to \kap$. 
	 
	\begin{defn}[{\cite[3.1]{PinkWedhornZiegler2}}] Let $ T $ be a scheme over $ \kap $. 
		\textbf{A $G$-zip of type} $\upmu$ over $T$ is a quadruple $\underline{I}= (I, I_+, I_-, \iota)$ consisting of a right $G$-torsor $I$ over $T$, a $P_+$-torsor $I_+\sbt I$, and $P_-^{\sig}$-torsor $I_-\sbt I$, and an isomorphism  of $M^{\sig}$-torsors: 
		$$\iota: I_+^{\sig}/U_{+}^{\sig}\cong I_-/U_-^{\sig}.$$
	\end{defn}
	A morphism $\underline{I} \to \underline{I}'=(I', I_+', I_-', \iota')$ of $G$-zips of type $\upmu$
	over $T$ consists of a $G$-equivariant morphism $I\to I'$ which sends $I_+$ to $I_+'$ and $I_-$ to $I_-'$,  and which is compatible with the isomorphisms $\iota $ and $\iota'$. 	
	The category of $ G $-zips over all $ \kap $-schemes form an algebraic stack over $ \kap $.

	For the cocharacter $\upmu$ there is an associated group scheme $E_{\upmu}\sbt P_+\times P_-^{\sig}$, called the \textbf{zip group} of $ \upmu $, which is given on points of a $\kap$-scheme $T$ by
	\begin{equation}\label{def of twisted Emu}
		E_{\upmu}(T)=\{(u_+m, \umin\sig(m))\ |\ m\in M(T),  u_+\in U_+(T), u_-\in U_-^{\sig}(T)\}. 
	\end{equation}
Here we use the decomposition of $\kap$-groups $P_+=U_+\rtimes M, \ P_-=U_-\rtimes M$.  Clearly we have an isomorphism of $\kap$-group schemes\[ U_+\rtimes M\ltimes U_- \cong \emu, \ \ (u_+, m, u_-)\mapsto u_+mu_-,\] where we omit describing the group law of the LHS. In particular, $\emu$ is a smooth connected linear algebraic group over $\kap$.  Consider its right action  on $G_{\kap}$ by 
	\begin{equation}\label{Grpactontorsors}
		g\cdot(p_+, p_-) = p_+^{-1}g p_-=m^{-1}u_+^{-1}g\umin \sig(m).
	\end{equation}
    With respect to this action one can form the quotient stack $[G_{\kap}/E_{\upmu}]$ over $\kap$. Here we use the right action while in  \cite{PinkWedhornZiegler1} and \cite{PinkWedhornZiegler2}, as well as in \cite{ChaoZhangEOStratification}, the left action is used, but apparently the resulting stacks $ [E_{\upmu}\backslash G_{\kap}] $ and $ [G_{\kap}/E_{\upmu}] $ are canonically isomorphic. 
	
	\begin{thm}[{\cite[ 3.11, 3.12]{PinkWedhornZiegler2}}]\label{algebraicity of zip stack} The stacks $G\textsf{-Zip}^{\upmu}$ and $[ G_{\kap}/E_{\upmu}]$ are naturally isomorphic. They are smooth algebraic stacks of dimension $0$ over $\kap$.  
	\end{thm}

	\subsection{The universal $ G $-zip over $ S $}\label{SectionDefinitonofEO}
	In this subsection we give definitions of those torsors appearing in the universal $G$-zip constructed in \cite{ChaoZhangEOStratification} and refer to loc. cit. for more details.
	
	For the relative de Rham cohomology $\HH^1_{\dR}(\mathcal{A}/S)$, apart from the well-known Hodge filtration $\omega_{\sA/S}\sbt \HH^1_{\dR}(\mathcal{A}/S)$, there is another filtration 
	\[
	\overline{\omega}_{\sA/S}:= \textbf{R}^1\pi_*\sH^0 (\Omega^{\bullet}_{\sA/S}), 
	\]
	called the \textbf{conjugate filtration} of $\HH^1_{\dR}(\mathcal{A}/S)$, fitting into the short exact sequence
	\begin{equation}\label{ConjFil}
		0\to \overline{\omega}_{\sA/S}\to  \HH^1_{\dR}(\mathcal{A}/S)\to \pi_* \sH^1 (\Omega^{\bullet}_{\sA/S})\to 0,  
	\end{equation}
	of locally free $\sO_{S}$-modules. This short exact sequence 
	is a particular consequence of the degeneration at $\rm E_2$-page of the conjugate spectral sequence \footnote{The denegeration of the conjugate spectral sequence at $\mathrm{E}_2$-page follows from that of the Hodge-de Rham spectral sequence at $\mathrm{E}_1$-page; see for example \cite[2.3.2]{KatzNilpConn}.} 
	\[
	{}_{\rm conj}\mathrm{E}^{r,s}_2:=\textbf{R}^{r}\pi_*\sH^{s} (\Omega^{\bullet}_{\sA/S})\implies \HH^{r+s}_{\dR}(\mathcal{A}/S).
	\]
	As discussed in \cite[7.1-7.5]{Moonen&WedhornDiscreteinvariants}, Cartier isomorphisms ((7.4) in loc. cit.) induces the following direct-summand-wise isomorphism of $\sO_S$-modules,\begin{equation}\label{ZIsomDeRham}	\updelta: \omega_{\sA/\sS}^{\sig}\oplus  \big(\HH^1_{\dR}(\mathcal{A}/S)/\omega_{\sA/\sS}\big)^{\sig}\cong \big(\HH^1_{\dR}(\mathcal{A}/S)/\overline{\omega}_{\sA/S}\big)\oplus \overline{\omega}_{\sA/S}. 	\end{equation}	We call the direct-summand-wise isomorphism $\updelta$ the \textbf{zip isomorphism} associated with the universal abelian scheme $\sA$ over $S$. The tuple $ \big(\HH^1_{\dR}(\sA/S), \omega_{\sA/\sS}, \overline{\omega}_{\sA/S}, \updelta\big)$ is  an ``$F$-zip" in the terminology of \cite{Moonen&WedhornDiscreteinvariants}. We call it the \textbf{universal $F$-zip} over $S$. The universal $G$-zip over $S$, to be defined below, should be viewed as the universal $F$-zip over $S$ with a $G$-structure.	\begin{rmk}\label{Rmk:ZipIsom}	The zip isomorphism $\updelta$ above can also be constructed using crystalline Dieudonn\'e theory, without explicit reference to Cartier isomorphisms, as is done in \cite{ChaoZhangEOStratification}. 	Indeed, there are canonical isomorphism of $\sO_S$-modules	$\HH^1_{\dR}(\mathcal{A}/S)\cong \Ds(\sA[\pinf])_{S}\cong \Ds(\sA)_S$, where $\Ds(\sA)_{S}$ is the restriction on  $S_{\rm Zar}$ (namely, the Zariski site of $S$) of the Dieudonn\'e crystal $\Ds(\sA)$ associated to $\sA$ (\cite[2.5.7]{BBM}). Under this canonical isomorphism, the Hodge filtration on both sides coincide (\cite[2.5.8]{BBM}) and the conjugate filtration $\overline{\omega}_{\sA/S}$ is equal to $\Ker(V: \Ds(\sA)_{S}\to \Ds(\sA)_{S}^{\sig})$. Then one can proceed to construct $\updelta$ in the same way as described in \S \ref{S:FzipandDieudonn} below.	\end{rmk}

	 Write $\Lmd_{\kap}^*=\Lmd_{ 0}^*\oplus \Lmd_{ -1}^*$ for the weight decomposition of $ \Lmd_{\kap}^{*} $ induced by \textbf{the inverse} of $ \upmu^{\sig}$.  Due to the canonical isomorphism $\Lams_{\kap}\cancong\Lmd_{\kap}^{*,\sig}$, such a decomposition  can be described in a  different way: if $\Lams_{\kap}=\LamsO\oplus \LamsI$ is the weight decomposition of $\Lmd^*_{\kap}$ induced by $\upmu$ as in \eqref{Weighdecomp}, we have 	\begin{equation}\label{ConjDecom}	\Lams_{ 0}=\can^{-1}(\Lmd^{*,0,\sig}), \ \ \	\Lams_{ -1}=\can^{-1}(\Lmd^{*,1,\sig}).   	\end{equation}	Here $\Lmd^{*,0,\sig}:=(\Lmd^{*,0})^{\sig}$; similarly for $\Lmd^{*,1,\sig}$.	Then $P_{-}^{\sig}$ is the schematic stabilizer in $G_{\kap}$ of the filtration $ \Lmd^{*}_{0}\sbt \Lams_{\kap}$. Now we come to the definitions of the following $ \kap $-schemes,																				\begin{align}	\begin{array}{l}	\mathrm{I}:=\Isom_{\sO_S}\big([\Lambda_{\kap}^*, \ 	s_{\kap}]\otimes \sO_{S},\ \ 	[\HH^1_{\dR}(\mathcal{A}/S),\ s_{\dR}]\big),\\	\rI_+:=\Isom_{\sO_S}\big([\Lmd^*_{\kap}\spt \Lmd^{*,1},\  s_{\kap}]\otimes \sO_{S},\ \ 	[\HH^1_{\dR}(\mathcal{A}/S)\spt \omega_{\sA/S}, \  s_{\dR}]\big),\\	\mathrm{I}_- :=\Isom_{\sO_S}\big([\ \Lmd_{ 0}^*\sbt \Lambda_{\kap}^*,\ \    s_{\kap} ] \otimes \sO_{S}, \ \ [\overline{\omega}_{\sA/S}\sbt \HH^1_{\dR}(\mathcal{A}/S), \  s_{\dR}]\big).	\end{array}	\end{align}																											Clearly $\rI$ and $\rI_+$ here are special fibres of $\II$ and $\II_+$ respectively (\S \ref{Torsors}). The group $G_{\kap}$ (resp. $P_+$, resp. $\Pminsig$ ) acts on $\mathrm{I}$ (resp. $\rI_+$, resp. $\rI_-$) on the right, as in \eqref{Grpactontorsors}.
	 
      \begin{thm}[{\cite[2.4.1, 3.1.2]{ChaoZhangEOStratification}}]\label{map zeta}	\begin{enumerate}[(1).]	\item The scheme $ \mathrm{I} $ (resp. $ \iplus$, resp. $  \mathrm{I}_{-}$) is a $ G_{\kap} $-torsor (resp. $ P_{+} $-torsor, resp. $ \Pminsig $-torsor) over $ S $.	\item The direct-summand-wise isomorphism $\updelta$ in \eqref{ZIsomDeRham} induces an isomorphism	\[	\iota: \mathrm{I}_+^{\sig}/U_+^{\sig}\cong \mathrm{I}_{-}/\Uminsig.	\]	Hence, the tuple  $\underI:=(\mathrm{I}, \iplus, \Imin, \iota)$ is a $G$-zip of type $\upmu$ over $S$, inducing a morphism of algebraic stacks over $ \kap$	\[ \upzeta: S\longrightarrow  \Gzips\cong \zipstack.\]	\item 	The map $ \upzeta $ is a smooth map of $ \kap $-stacks. 		\end{enumerate}	\end{thm}
	
      We shall call the tuple $\underI:=(\mathrm{I}, \iplus, \Imin, \iota)$  the \textbf{universal $G$-zip over $S$} and $ \upzeta $ the \textbf{zip period map for $S$}. As indicated at the beginning of this section, our focus  in this paper is the map $\upzeta$ itself. But for the reader's curiosity, we end this section by giving the definition of EO strata for $S_{\bar{\mathbb{F}}_p}$. 	\begin{defn}[{\cite[3.1.1]{ChaoZhangEOStratification}}]	\label{definition of EO strata} Set $k=\Fpbar$. 	For a geometric point $ w: \Spec k\to G\textsf{-Zip}^{\upmu}$, the \textbf{EO stratum} of $S_k$ associated to $ w $, denoted by $ S_k^{w} $, is defined to be the fibre of $ w $ under the zip period map $\upzeta_k: S_k\to G\textsf{-Zip}^{\upmu}_{k}$. 	\end{defn}

	Merely by definition of $\upzeta$, being a morphism of algebraic stacks, and the property of $[G_{\kap}/E_{\upmu}]$ being a $0$-dimensional stack, one learns that each $ S_{k}^{w} $ is a locally closed subscheme of $ S_{k} $. Moreover, the smoothness of $ \upzeta $ implies that each $ S_{k}^{w}  $ is automatically a smooth $ \kap $-scheme. See \cite{ChaoZhangEOStratification} for more properties of these EO strata.

	\section{\protect{Construction of $\upgamma: \iplus\to G_{\kap}/\Uminsig$}}\label{S:ConstructionofEta}

	The  main goal of this section is to construct a morphism of $ \kap $-schemes $ \upgamma: \iplus\to\gmoduminsig $ and to deduce from it a morphism of $\kap$-stacks, $ \upeta: S\to \zipstack $. The comparison of $\upeta$ with~$\upzeta$ will be given in \S \ref{S: CompaZipMaps}. Here  $\gmoduminsig$ is the quotient fpqc sheaf of $ \Gka $ by the $ \Uminsig $-action  via right multiplication. It is represented by a scheme, smooth separated of finite type over~$ \kap $, and the canonical projection $\Gka \to \Gka/\Uminsig$ is smooth; see for example~\cite[7.15,~7.17]{MilneAlgebraicGroups}.

	\subsection{Trivialized Frobenius}\label{S: TriviaFrob}
	Let $ \barR $ be a $ \kap $-algebra which Zariski locally admits a finite $p$-basis, and $\xbarflat=(\xbar,\betxbar)\in \mathrm{I}_+(\bar R)$ an $\bar R$-point of $\mathrm{I}_+$; (cf. \S \ref{Torsors}). Let $\underline{R}=(R, \sig)$ be a simple frame of $\bar R$ (which exists by Lem. \ref{Froblifts}) and $\xflat=(x, \betx)\in \II_+(R)$ a lift of $\xbarflat$; here $\xflat$ exists since $\II_+$ is smooth over $\wkap$.
	
	 By Thm. \ref{ClaBT}, the $p$-divisible group $\sA_x[p^{\infty}]$ corresponds to an object in $\mathbf{AFDM}(\underline{R}, \nabla)$, namely an admissibly filtered Dieudonn\'e module over $\underline{R}=(R, \sig)$,
	\[
	\underline{M}=(M, \FF, \VV, \nabla_M, M^1_x),\ \  \text{with}\ \ M= \Ds(\sA_x[\pinf])(R),
	\]
	where, with the simple frame $\underline{R}$ fixed, the Dieudonn\'e module $(M,\FF, \VV,\nabla_M)$ is determined by $\sA_{\bar x}[\pinf]$, and hence by $\bar x$, while the admissible filtration  $M^1_x\sbt M$ depends on the lift $x$ of $\bar x$. In particular, depending on the objects ($x$ or $\xbar$) we want to emphasize, we can write \[\FF=\FF_{x}=\FF_{\xbar}: \Msig\to M,\]  By Rem.  \ref{AbComp} that we have canonical isomorphism of filtered $R$-modules\begin{equation}\label{CanIsom}	\big( M\spt M^1_{x}\big) \cancong \big(\HH^1_{\dR}(\sA_x/R) \spt \omega_{x}\big).	\end{equation} For this reason we identify them and this identification  equips $ M $ with a set of tensors $ s_{\dR, R}\in M^{\otimes}$.
	With this identification we view $\betx$ as a trivialization of the filtered module~$\big(M\spt M^1_x\big)$, \[	\betx: \big(\Lams_R\spt \LamsI_R\big)\cong \big(M\spt M^1_x\big).	\]  Note that since $ \Lams $ is a free $ \mathbb{Z}_{(p)} $-module, we have  canonical isomorphisms~$\eps: (\Lambda_{R}^*, s_{R}) \cancong (\sig^* \Lambda_{R}^*, \ \sig^{*}s_{R})$. 	By transport of structure we obtain the \textbf{trivialized Frobenius} 	\begin{equation}	\FF_{\xflat}=\betx^{-1} \FF  \sigbetx: \Lams_R\to \Lams_R,	\end{equation}	where we set	$\sigbetx: \Lams_R \to M^{\sig}$ 	to be $\betxsig  \eps$.  For an element $h\in \sP_+(R)$,  we have	\begin{equation}\label{ConjFrob}	\FF_{\xflat\cdot h}=h^{-1}\FF_{\xflat} \sig(h),    	\end{equation}	by definition of the action of $\sP_+(R) $ on $ \II_+(R) $ (\S \ref{Torsors}). Here, $\sig(h)$ is defined as,  \begin{equation*} \sig(h):=\eps^{-1} h^{\sig} \eps;	\end{equation*}	this coincides with our notational convention \eqref{MixFrob}.  Sometimes, we simply identify $\sig(h)$ and $h^{\sig}$ by suppressing the canonical isomorphism $\eps$ above. Clearly, for an element $\xflat\in \II(R)$, we can define $\FF_{\xflat}$ in the same way.

	\subsection{Frobenius invariance of tensors}\label{S: Frobinvar}
	The setting in this subsection is the same as in the previous subsection \S \ref{S: TriviaFrob}.
	\begin{lem}\label{Lem:FrobInvariant}
	For an element $\xflat=(x, \betx)\in \II(R)$, the Frobenius 
		$\FF_{\xflat}$ defined above preserves tensors $ s_{R} $ termwise. In particular, we have 
		\[
			\FF_{\xflat} \in \sG(R[\frac{1}{p}]).
		\]
	\end{lem}
	\begin{proof}
		This follows from  the next lemma  and the definition of $\FF_{\xflat}$. 
	\end{proof}

	\begin{lem}
		The Frobenius $\FF: \Msig\to M$, after inverting $ p $, sends $\sig_{R}^*s_{\dR, R}$ to $s_{\dR,R}$ termwise.
	\end{lem}
	\begin{proof}
	 For any maximal idea $ \mathfrak{m} $ of $ R $,	by Lem. \ref{Froblifts}, (4) the Frobenius lift $ \sig: R\to R $ induces a simple frame ($ \hat{R}_{\mathfrak{m}}, \sig )$ of $  \hat{R}_{\mathfrak{m}} $, compatible with $ (R, \sig_R) $. Note that $\hat{R}_{\mathfrak{m}}$ is necessarily $ p $-complete ( since $  \mathfrak{m}  $ contains $ p $). Hence it suffices to show the lemma after base change to $\hat{R}_{\mathfrak{m}}$ for all $ \mathfrak{m} $. In particular, we may assume that $ R $ is a local ring. 
		
		Let $ s'\in \mathcal{S} $ be the image of the closed point of $ \Spec R $, which necessarily lies in the special fibre $ S\sbt \mathcal{S} $. Let  $s\in  S $ be a closed point which is a specialization of $ s' $. Then the morphism $ x: \Spec R\to \mathcal{S} $ facts through the canonical embedding $ s: \Spec A \to \sS$, where $A:=\hat{\sO}_{\sS, s}$ is the complete local ring of $\sS$ at $s$.  Choose a $ W(k)$-isomorphism $A\cong W(k)[[X_1,\cdots, X_r]]$ and consider the Frobenius lift $\sig_{A}: A\to A$ of $A$ given by sending each $X_i$ to its $p$-th power.   Write $\underline{N}:=(N, \FF_N, \VV_N, \nabla_{N})$ for the Dieudonn\'e module over $(A,\sig_{A})$ of $\sA_{s}[p^\infty]$. Then the induced de Rham tensor $s_{\dR, A}\in N^{\otimes}$ is horizontal. We claim that $s_{\dR, A}$ is also Frobenius invariant.  Prior to showing the claim, let us note that the claim implies our lemma. Indeed, if we let $f: A\to R$ to denote the structure morphism, then $M$ is canonically isomorphic to the pull back $f^*N=N\otimes_{A,f}R$. If we identify this canonical isomorphism, then the Frobenius $\FF_M$ is equal to 
		\[\Msig\cancong \sig_R^*f^*N\overset{\upepsilon}{\cong} f^* \sig_{A}^*N\xrightarrow{f^*\FF_N} f^*N\cancong M,	\]	
		where the isomorphism $\sig_R^*f^*N\overset{\upepsilon}{\cong} f^* \sig_{A}^*N$ is provided by the integrable connection $\nabla_{N}$ (note that $ \sig_R\circ f $ and $f\circ \sig_{A}$ become the same after modulo $p$); in fact, due to our choice of free coordinates $X_i$, it is possible to give an explicit formula for $\upepsilon$ (see, for example \cite[1.5]{KisinIntegralModels}). Then since $s_{\dR, A}$ is horizontal, one sees that $\upepsilon$ sends $\sig_{R}f^* s_{\dR, A}$ to $ f^*\sig_{A}^*s_{\dR, A}$ (this can also be seen from the explicit expression of $\upepsilon$, cf. \cite[1.5.4]{KisinIntegralModels}). Now it is clear that we are reduced to show the claim.
		
		The proof of the claim is already in the proof of \cite[Prop. 2.3.5]{KisinIntegralModels}. We now write $B$ for the adapted deformation ring $R_{G_W} $ in loc. cit. of the $ p $-divisible group $\sA_{\bar s}[p^{\infty}]$ over $ k:=k(s) $, where we use $\bar s:\Spec k\to S$ to denote the special fibre of $s$. Again the $\wkap$-algebra $B$ is isomorphic to some power series ring over $W(k)$ and one can equip it with a Frobenius lift $\sig_B: B\to B$ by sending free coordinates to their $p$-th powers. Write $\underline{L}:=(L, \FF_L, \VV_L, \nabla_L)$ for the Dieudonn\'e module over $(B, \sig_{B})$ that corresponds to the universal $p$-divisible group over $B$; the construction of $\underline{L}$ is explained in \cite[1.5.4]{KisinIntegralModels}.  By construction the Dieudonn\'e module $L$ comes with \textbf{Frobenius-invariant} tensors which we denote by $s_{\cris}$. We know from the proof of \cite[1.5.4]{KisinIntegralModels} (see also \cite[Thm. 3.3.12]{LoveringFilCryst} for more details) that there exists a $\wkap$-algebra homomorphism (in fact an isomorphism) $g:B\to A$ such that the tuple $\big(\underline{N}, s_{\dR, A}\big)$ is obtained as the pull back along $g$ of the tuple $\big(\underline{L}, s_{\cris}\big)$, expect that one has to use the integrable connection $\nabla_L$, as we did for $\nabla_{N}$,  to deal with the possible incompatibility of Frobenius lifts between $\sig_A$ and $\sig_{B}$. Again using the fact that  $s_{\cris}$ is horizontal, we conclude that $s_{\dR, A}$ is also Frobenius invariant. 
	\end{proof}

	\subsection{Trivialized partially divided Frobenius}
	Let $\Rbar, \underline{R}$, $\xbarflat$ and $\xflat$ be as in \S \ref{S: TriviaFrob}.
	For the next lemma, which plays important roles for our construction  of $ \gam: \iplus\to G_{\kap}/\Uminsig $ below,  one may recall the characters	$ \mutil $ and  $ \mutilsig=\sigmutil $ in \S \ref{S:cochar}.

	\begin{lem}\label{IntegralLem}
		For every $\xflat\in \II_+(R)$ we have $\FF_{\xflat}\in\mathcal{G}(R)\mutilsig(p)\sbt  \sG(R[\frac{1}{p}])$; 
		i.e.,  we have 
		\[
			\FF_{\xflat}=\Intxflat\mutilsig(p)
		\]
		for some (necessarily unique) element $\Intxflat\in \sG(R)$.
	\end{lem} 
	
	\begin{proof}
		The weight decomposition 
		$\Lams_R=\Lambda_R^{*,1}\oplus \Lambda_R^{*,0}$ given by $\mutil$ (see \eqref{Weighdecomp}), induces via $\betx$ a normal decomposition $M=M^{1}\oplus M^{0}$ of $M$. 
		Then by \S \ref{ParDivFrob} we have the decomposition \[\FF_{x}= \Gamxflat \circ \rfxflat,\]
		with $\Gamxflat$ and $\rfxflat$ defined as in \eqref{defGamma}. Note that by Lem. \ref{Zipisom}, the partially divided Frobenius $\Gamxflat$ is an isomorphism of $ R $-modules.
		Now we have
		\begin{equation}\label{Eq:Fxflat}
			\FF_{\xflat}=   \betx^{-1} \FF_{\xbar} \sigbetx	= \betx^{-1}\big( \Gamxflat \rfxflat\big)\sigbetx 	=\Intxflat \big(\sigbetx^{-1} \rfxflat \sigbetx\big),\end{equation}	where $\Intxflat$ is defined as \begin{equation}	\Intxflat:=\betx^{-1} \Gamxflat \sig(\betx).\end{equation}   	Clearly we have $\Intxflat\in \GL(\Lams_R)$. Unwinding the definition of $  \mutilsig$ in \S \ref{S:cochar} we see that	\[	\sig(\betx)^{-1}\rfxflat \sig(\betx)=\mutilsig(p).	\]
		Now the equality \eqref{Eq:Fxflat} becomes  $\FF_{\xflat}=\Intxflat\mutilsig(p)$, and hence by Lem. \ref{Lem:FrobInvariant} we have
		\[
		\Intxflat\in \sG(R[\frac{1}{p}])\cap \GL(\Lams_R)=\sG(R).
		\]
		
	\end{proof}
	We will call $\Intxflat\in \sG(R)$ the \textbf{trivialized partially divided Frobenius} attached to $\xflat$, w.r.t. the simple frame $\underline{R}=(R, \sig)$.

	\subsection{Local version of $\gam$} \label{LocConstr} We continue our discussion in the setting of the previous subsection. 
	
	With the simple frame $\underline{R}$ fixed, for an element $\xflat\in \II_+(R)$, as usual we write $ \overline \Intxflat \in G(\bar R)$ for the reduction modulo $p$ of the trivialized partially divided Frobenius $ \Intxflat $. Denote by $\gamxbarflat\in G_{\kap}/\Uminsig(\barR) $ for the image of $\Intxflatbar$ along the canonical projection $ G(\barR)\to  G_{\kap}/\Uminsig(\barR)$. The notation $\gamxbarflat$ is justified by the following proposition.

	\begin{prop}\label{Locconstr}
		The element $\gamxbarflat\in G_{\kap}/\Uminsig(\barR) $ is determined by $\xbarflat$, i.e., it is independent of the choice of lifts $\xflat$ and the choice of simple frames $\underline{R}$ of $\bar R$.  In particular, it induces a morphism of $ \kap $-schemes,
		\[ \gamxbarflat: \Spec \barR\to G_{\kap}/\Uminsig. \]
	\end{prop}
	
	\begin{proof}
		Let $ \underline{R'}=(R', \sig')$ be another simple frame of $\bar R$ and $\yflat=(y, \bety)\in \II_+(R')$ another lift of $\xbarflat$. Denote by $\smallint_{\yflat, \sig'}\in\sG(R')$ be the trivialized partially divided Frobenius attached to~$\yflat$, w.r.t. $\udl R'$. We will compare below the two elements $\Intxflatbar, \overline{\smallint_{\yflat, \sig'}}\in G(\Rbar)$. 
		
		Note first that we may assume $R=R'$. Indeed, by Lem. \ref{Froblifts} we can choose an isomorphism of $\wkap$-algebras $\epsilon: R'\cong R$ whose reduction modulo $p$ is $\id_{\Rbar}$. This isomorphism induces anther frame structure, 
		\[(R, \sig'):=\epsilon^*(R', \sig')\] on $R$. 
	    Denote by $\smallint_{\epsilon^*(\yflat), \sig'}$ the trivialized partially divided Frobenius  of $\epsilon^*(\yflat)\in \II_+(R)$ , w.r.t. simple frames $(R, \sig')$. Then clearly we have 
		$\overline{\smallint_{\yflat, \sig'}}=\overline{\smallint_{\epsilon^*(\yflat), \sig'}}$.
		
		From now on, we assume $R=R'$ and denote by $\smallint_{\yflat}\in\sG(R)$ be the trivialized partially divided Frobenius attached to~$\yflat$, w.r.t. $\udl R$. Now the proposition follows from the combination of Lem. \ref{Glue1} below which compares $\Intxflatbar $ and $ \Intyflatbar$, and Lem. \ref{Glue2} below which compares $\Intyflatbar$ and $\overline{\smallint_{\yflat, \sig'}}$. 
	\end{proof}
	
	\begin{lem}\label{Glue1}
		With the simple frame $\udl R$ fixed, if $\yflat=(y, \bet_y)\in \mathbb{I}_+(R) $ is another lift of $ \xbarflat $. Then there exists  $ \umin\in \Umin (\bar R)$, such that following holds in $ G(\bar{R}) $
		\begin{equation}
			\overline{\Intxflat}= \overline{\Intyflat}\cdot \sig(\umin).
		\end{equation}
	\end{lem}
\begin{proof} Recall that by \S \ref{CrynatrofdRTensor} we have the canonical parallel  isomorphism \[ \epsilon(x,y):\HIdR(\sA_x/R)\cancong \Ds(\sA_{\bar x})(R)\cancong \HIdR(\sA_y/R), \]  which carries $s_{\dR, x}$ to $s_{\dR, y}$.  Hence for our purpose we may assume $x=y$; note however that this does NOT mean\footnote{One may instead write $\bet_{x}'$ for $\bet_y$ in the discussion below, but for typographical reason we choose not to do so.} that $\bet_x=\bet_y$.  Denote by $\II_x$  the  trivial $\sG$-torsor over $\Spec R$, obtained as the pullback to $\Spec R$ of $\II$ along $x:\Spec R\to \sS$ and view $\betx,\bety$  as  elements in $\II_{x}(R)$. Write $ h:=\bety^{-1}\circ \bet_{x}\in \sG(R)$. By \eqref{ConjFrob} we have 	\[ \Intxflat=h^{-1}\Intyflat\cdot \sig\big(\mutil(p)h\mutil(p)^{-1}\big).	\]
		\label{Need to add more details here} 
	   Hence it suffices to show the following 
		\begin{align}\label{InUmin}
			\tilde{\upmu}(p)h\tilde{\upmu}(p)^{-1}\in \sG(R), \ \ \ \overline{\tilde{\upmu}(p)h\tilde{\upmu}(p)^{-1}} \in U_{-}(\barR).
		\end{align}
		To show these we use the embedding $\iota: \sG_{W(\kap)}\hookrightarrow \GL_{2\rg, W(\kap)}$ introduced in \S \ref{S:Emdcoord}. Since we have   $ \mutil(p)h\mutil(p)^{-1} \in  \sG(R[\frac{1}{p}]) $, in order to show that it lies in $\sG(R)$, it suffices to show that it lies in $\GL_{2\rg, \wkap}(R)$. Moreover, by \eqref{Intersect}, in order to show $\overline{\tilde{\upmu}(p)h\tilde{\upmu}(p)^{-1}} \in U_{-}(\barR)$, we may replace $\mutil$ by the induced cocharacter $\mutil'=\iota\circ \mutil $ of $\GL_{2\rg, \wkap}$. Inside $\GL_{2\rg}(R[\frac{1}{p}])$, $\mutil'(p)$ and $h$ (note that $\bar h=1$) are represented by matrices of the following forms respectively 
		\begin{equation} \label{MatrxRep}
			\Big( \begin{array}{cc}
				p{\rm I}_{\rg} & \\
				&{\rm I}_{\rg}\end{array}\Big)  \text{\ \ \ and \ \ \ } \Big( \begin{array}{cc}
				{\rm I}_{\rg}+pA &pB \\
				pC&{\rm I}_{\rg}+pD \end{array}\Big), 
		\end{equation}
		where $ A, B, C, D $ are $ \rg$ by $ \rg $ matrices with entries in $ R $. Now the problems become trivial due to our discussion at the end of \S \ref{S:Emdcoord}:\begin{equation}\label{MatrixCal}	\Big( \begin{array}{cc}
			p{\rm I}_{\rg} & \\
			&{\rm I}_{\rg}\end{array}\Big)  \Big( \begin{array}{cc}
			{\rm I}_{\rg}+pA &pB \\
			pC&{\rm I}_{\rg}+pD \end{array}\Big) \Big(\begin{array}{cc}
			p{\rm I}_{\rg} & \\
			&{\rm I}_{\rg}\end{array}\Big)^{-1} = \Big( \begin{array}{cc}
			{\rm I}_{\rg}+pA &p^2B \\
			C&{\rm I}_{\rg}+pD \end{array}\Big). \end{equation}
	\end{proof}

	\begin{lem}\label{Glue2} 
		Fix a lift  $\xflat \in \II_+(R)$ of  $\xbarflat$ and let $ \underline{R'}=(R, \sig')$ be another simple frame of $\bar R$.
		Then there exists an element $ u_{- }\in U_{-}^{\sig} (\bar R)$, such that 
		\begin{equation}
			\overline{\Intxflat}= \overline{\smallint_{\xflat, \sig'}}\cdot u_{-}
		\end{equation}
	\end{lem}
	\begin{proof} By basic properties of the Dieudonn\'e crystal  $\Ds(\sA_{\xbar})$ (cf. also \S \ref{CrynatrofdRTensor}), we have canonical parallel isomorphism $ \iota: \sig^*M\to \sig'^* M$, such that 	$ \FF=\FF'\circ \iota $. By direct computation one sees that 	                  \[	\Intxflat=\smallint_{\xflat, \sig'}\big( \mutilsig(p)h\mutilsig(p)^{-1}\big), 	\]	where 	$h:=\sig'(\betx)^{-1}\iota \sig(\betx)\in \sG(R)$, 	and the superscript $``\sig"$ in $ \mutilsig$ certainly refers to the Frobenius lift $\sig: W(\kap)\to W(\kap)$. 	We use again  the embedding $\iota: \sG_{W(\kap)}\hookrightarrow \GL_{2\rg, W(\kap)}$ in \S \ref{S:Emdcoord}, but in a twisted manner. To be precise, the   pull back of $\iota$ along $\sig: W(\kap)\to W(\kap)$ induces another embedding   	\[	\sig(\iota): \sG\cancong\sG^{\sig}\xrightarrow{\ \iota^{\sig}} \GL_{2\rg, W(\kap)}^{\sig} \cancong\GL_{2\rg, W(\kap)}.	\]	Exactly as in the proof of Lem. \ref{Glue1}, it suffices to show $\overline{\mutilsig(p)h\mutilsig(p)^{-1}} \in U'^{,\sig}_{-}(\barR)$, with $U'^{,\sig}_{-}\sbt \GL_{2\rg, \kap}$ the counterpart of $\Uminsig\sbt \Gka$ for the cocharacter		\[	\Gmka\cancong\Gmka^{\sig}\xrightarrow{\upmu'^{,\sig}} \GL_{2\rg, \kap}^{\sig}\cancong\GL_{2\rg, \kap} \]	Via the embedding $\sig(\iota)$,  $\mutilsig(p)$ and $h$ are represented inside $\GL_{2\rg}(R[\frac{1}{p}])$  by matrices of the same forms as in \eqref{MatrxRep} respectively, and $U'^{,\sig}_-(\bar R)$ consists of matrices of the form 	$\Big(\begin{array}{cc}	\rm{I}_{\rg}&\\		*& \rm{I}_{\rg}	\end{array}\Big)$.	Now we finish by the same calculation as in \eqref{MatrixCal}.
	\end{proof}

	\subsection{The global map $\gam: \rI_+\to \GmodUminsig$ via gluing} \label{S:Glue} 
	
	   In this subsection we apply Prop. \ref{LocConstr} to construct the global map $\gam: \rI_+\to \GmodUminsig$. For this we take a Zariski affine open covering $ \{\xbarflat_i:\Spec\bar R_i\hookrightarrow \rI_+\}$ of $\rI_+$. As each $\Rbar_i$ is smooth over $\kap$, by Exam. \ref{ExofSimpFram},  Zariski locally it admits a finite $p$-basis, and hence we can apply Prop. \ref{Locconstr} and obtain morphisms of $\kap$-schemes,
	\begin{equation}\label{setoflocal morphism}
		\upgamma_i=	\upgamma_{\bar{x}_i^{\flat}}: \Spec \bari \to G_{\kap}/\Uminsig.
	\end{equation}
	
	\begin{thm}
		\label{MainThm1}
		The maps $ \gam_{i}$ defined in (\ref{setoflocal morphism}) glue to a map of $\kap$-schemes,
		$$ \gam: {\rm I}_{+}\longrightarrow G_{\kap}/U_{-}^{\sig}. $$
	\end{thm}
   	\begin{proof}
	   Since $ \rI_+ $ is quasi-projective (hence separated), the intersection of $\Spec \bar R_i$ and $\Spec \bar R_j$ is again affine. Denote it  by $ \Spec \bar R_{ij} $. We need to show $\gam_i$ and $\gam_j$ restrict to the same map on $\Spec \bar R_{ij}$, for all $ i ,j$. But since $\Rbar_{ij}$ is again a smooth $\kap$-algebra (hence Zariski locally admits finite $p$-basis), this follows from the next lemma, Lem. \ref{Lem:coin}.
    \end{proof}
	
	 \begin{lem}\label{Lem:coin}
		Given two $\kap$-algebras $ \Rbar, \Rbar'$ which Zariski locally admit finite $p$-basis,  and a morphism of $\kap$-schemes  $\xbarflat: \Spec \Rbar \to  \rI_+$, then for any morphism of $\kap$-schemes $\xi: \Spec \Rbar'\to \Rbar$, we have \[\gam_{\xbarflat\circ \xi}=\gam_{\xbarflat}\circ \xi.\]
	\end{lem}
	\begin{proof}
		This is immediate from our construction in Prop. \ref{LocConstr}, if there exists a homomorphism of simple frames $f:(R', \sig)\to (R, \sig)$ which lifts the structure map $\Rbar\to \Rbar'$. In general we do not know such an $f$ always exists; below we proceed by reducing a general case to cases where $f$ exists by ``passing to perfection". 
		
		Note first that for each $\Fp$-algebra $\bar{A}$ which Zariski locally admits a $p$-basis, the absolute Frobenius map $\sig: \bar{A}\to \bar{A}$ is faithfully flat (the $p$-basis assumption implies that $\bar{A}$ as an $\bar{A}$-module via $\sig$, is locally free). Consequently the canonical ring map  \[\Abar\to \Abar_{\perf}:=\varinjlim_{a\mapsto a^p}\bar{A}\] is faithfully flat. In particular, horizontal arrows in the following commutative diagram, 																\[																														  \xymatrix{\GmodUminsig(\Rbar)\ar[d]\ar[r]& \GmodUminsig(\Rbar_{\perf})\ar[d]\\	\GmodUminsig(\Rbar')\ar[r]&\GmodUminsig(\Rbarperf'),}
		\]  
		are injective. Now since the formation of $W(\cdot)$ is functorial, we are reduced to show $\gam_{\xbarflat\circ \pi}=\gam_{\xbarflat}\circ \pi$, with $\pi: \Spec \Rbar_{\perf}\to \Rbar$  the canonical morphism. This follows from the following fact: there is a sequence of homomorphisms of simple frames over $\wkap$, \[(R, \sig)\to (R,\sig)_{\perf}:=\big(\widehat{R_{\perf}}, \sig\big)\cong \big(W(\Rbarperf), \sig\big),\]
		which lifts the structure map $\Rbar\to \Rbarperf$ (see also \cite[Lem. 6.12]{Yan18} for another construction). We need to explain this fact:  $R_{\perf}$ is defined as the colimit perfection,\[R_{\perf}:=\varinjlim_{\sig:R\to R}R, \]
		and $\widehat{R_{\perf}}$ is the $p$-completion of $R_{\perf}$. Clearly we have $R_{\perf}/pR_{\perf}=\Rbarperf$. The Frobenius lift $\sig_R: R\to R$ induces a Frobenius lift $\sig$ on $R_{\perf}$ (hence on $\widehat{R_{\perf}}$) compatible with $\sig_R$, and hence we get the  homomorphism $(R, \sig)\to (R,\sig)_{\perf}$ of simple frames displayed above. In fact $R_{\perf}$ is $p$-torsion free and the simple frame ($R_{\perf}, \sig$), viewed as a crystalline prism (see \ref{Rmk: FramPrism}), is nothing but the perfection of the prism $(R, \sig)$ in the sense of Bhatt-Scholze \cite[Lem. 3.9]{PrismBS} and hence we justified the isomorphism of simple frames $ (R, \sig)_{\perf}\cong \big(W(\bar{R}_{\perf}), \sig\big) $; see \cite[Cor. 2.31]{PrismBS}.	
	\end{proof}

   	\subsection{The zip period map $ \upeta $}\label{S: ZipMapEta}
	The natural embedding $\Uminsig\hookrightarrow \emu$ realizes $\Uminsig$ as a normal subgroup of $\emu$. Via this embedding $\Uminsig$ acts on $\Gka$ by right multiplication. Passing to quotient, we obtain an action of $ P_{+}=\emu/\Uminsig $ on $ \FlG$ given on local sections by 
	$ g\cdot p_{+}=p_{+}^{-1}g\sig(m)$, 
	where $ p_{+}=u_{+}m $, with $ u_{+}\in U_{+} $ and $ m\in M $. Denote by $[(\FlG)/P_{+}]$ the resulting quotient algebraic stack over $\kap$. Since the action of $\Uminsig$ on $\Gka$ is free, the canonical projection $\Gka\to \FlG$ induces a canonical \textbf{isomorphism} of algebraic stacks over $\kap$,
	\[ 
	\zipstack\cong [(\FlG)/P_{+}].
	\]

	\begin{thm}\label{MainThm2}
		The map $ \upgamma $ is equivariant w.r.t. the actions of $ P_+ $ on $ \iplus $ and on $ G_{\kap}/U_{-}^{\sig} $, and hence induces a morphism of algebraic stacks over $ \kap $,
		\begin{equation*}
			\upeta: S\cong \iplus/P_{+}\to  \zipstackI \cong \zipstack\cong \Gzips. 
		\end{equation*} 
	\end{thm}
	
	\begin{proof}
		We need to show the commutativity of the following diagram of $ \kap $-schemes
		\begin{equation*}
			\xymatrixcolsep{5pc}\xymatrix{	\iplus\times_{\kap}P_{+}\ar[r]^{\ \upgamma\ \times\ \id_{P_+}}\ar[d]&(G_{\kap}/\Uminsig)\times_{\kap}P_{+}\ar[d]\\
				\iplus\ar[r]^{\upgamma}&G_{\kap}/\Uminsig,}
		\end{equation*}
		where vertical arrows are given by $ P_+ $-actions. 
		Since $\rI_+\times_{\kap}P_+$ is geometrically reduced, it suffices to check the commutativity on $ k $-points for an algebraically closed field extension $k$ of $\kap$. 	Note first that for any $ \xbarflat\in \rI_+(k) $, by Lem. \ref{Lem:coin}, we have 
		$ \gam(\xbarflat)=\gam_{\xbarflat}.$
		For any $k$-point 
		$(\xbarflat, \bar{p}_+)$ of   $\rI_+\times_{\kap}P_+$, 
		take a $W(\kap)$-point ($\xflat, p_+$) of 
		$ \II_+ \times_{W(\kap)}\sP_+$, which lifts $(\xbarflat, \bar{p}_+) $. 
		Then $\xflat \cdot p_+$ is a  lift of $\xbarflat\cdot \bar{p}_+$. Applying the construction in \S \ref{LocConstr}, we obtain an element  $\smallint_{\xflat\cdot p_+}\in \sG(W(\kap)) $. 
		A direct calculation using the relation \eqref{ConjFrob} gives the following 
		\[
		\smallint_{\xflat\cdot p_+}= p_+^{-1}\smallint_{\xflat}\big(\mutilsig(p)\sig(p_+)\mutilsig(p)^{-1}\big)=p_+^{-1}\smallint_{\xflat} \sig\big(\mutil(p)\uplu\mutil(p)^{-1}\big)\sig(m), \]
		where $ p_+=u_+m$,  with $ u_+\in \sU_+(W(k))$ and $m\in \sM(W(k)) $, and
		where for the second $``="$ one uses the fact that $m$ commutes with $\mutil(p)$ and that $\mutilsig(p)=\sig(\mutil(p))$. 
		But by Lem. \ref{Lem:integral}, the element $ \mutil(p)\uplu\mutil(p)^{-1} \in \sG(W(k)[\frac{1}{p}])$  actually lies in $ \sG(W(k)) $ and we have $\overline{\mutil(p)\uplu\mutil(p)^{-1}}=1\in G(\bar R)  $. 
	\end{proof}

	\section{Comparison of $ \upeta $ with $ \upzeta $}\label{S: CompaZipMaps}
	In this section we show that the map $ \upeta : S\to \Gzips$ constructed in Thm. \ref{MainThm2} coincides with the map $ \upzeta:S\to \Gzips $ in  \cite{ChaoZhangEOStratification}, in the sense that there are naturally 2-isomorphic. The strategy is to show that there is a natural isomorphism between their corresponding objects in the groupoid $\zipstack(S)$.

	\subsection{Zip isomorphisms associated with \protect{Dieudonn\'e} modules}\label{S:FzipandDieudonn}
	As a preparation for the next subsection, as in \S \ref{S: TriviaFrob} we let $\Rbar$ be a $\kap$-algebra which Zariski locally admits a finite $p$-basis, and choose a simple frame $\underline{R}=(R,\sig)$ for $\Rbar$.   Take a point $\xbar\in S(\Rbar)$ and denote by $\underline{M}=(M, \FF, \VV, \nabla)$  the Dieudonn\'e module over $\underline{R}$ that is associated with the $p$-divisible group $\sA_{\xbar}[p^{\infty}]$. Write $\bar \FF: \bar{M}^{\sig} \to \bar M, \bar \VV: \Mbar \to \bar M^{\sig}$ for the reduction modulo $p$ of $\FF, \VV$ respectively; note however that $\Mbar, \bar{\FF}, \bar{V}$ are independent of the choice of $\underline{R}$, as they can be obtained by taking evaluation  at the trivial PD thickening $\Rbar\xrightarrow{\id}\Rbar$ of the Dieudonn\'e crystal $\Ds(\sA_{\xbar})[p^\infty]$; see \S \ref{S:BT/Rbar}.   Then the relations  $\FF\circ \VV = p\cdot \id_{M}$ and $\VV\circ \FF = p\cdot\id_{M^{\sig}}$ give rise to an exact sequence of $\Rbar$-modules
	\begin{equation*}
		\Mbarsig\xrightarrow{\FFbar} \Mbar \xrightarrow{\VVbar} \Mbarsig\xrightarrow{\FFbar} \Mbar.
	\end{equation*}
	And hence canonical isomorphisms $\FFbar: \Mbarsig/\Ker(\FFbar)\xrightarrow{\cong} \Ker(\VVbar),[\VVbar]:  \Mbar/\Ker(\VVbar)\xrightarrow{\cong}\Ker (\FFbar)$; 
	combining them, we obtain a canonical direct-summand-wise isomorphism of $\Rbar$-modules
	\begin{equation}\label{ZipIso1}
		\updelta: \Ker(\FFbar)\oplus \Mbarsig/\Ker(\FFbar)\xrightarrow{\  [\VVbar]^{-1}\oplus \FFbar \ } \Mbar/\Ker(\VVbar) \oplus \Ker(\VVbar).
	\end{equation} 
	
      We call $\updelta$ above the \textbf{zip isomorphism} associated with the Dieudonn\'e module $\Mbar$.  Now we make connection to the zip isomorphism we defined in \eqref{ZIsomDeRham} (cf. Rem. \ref{Rmk:ZipIsom}). Let $\MbarI\sbt \Mbar$ be the Hodge filtration of $\Mbar$ as introduced in \S \ref{BT/R}. As recalled in \eqref{HodgeFil}, we have $\MbarIsig=\Ker(\FFbar)\sbt \Mbarsig$.  We identify the following  canonical isomorphism, 
   	\begin{equation*}
		\big( \Mbar\spt \MbarI\big) \cancong \big(\HH^1_{\dR}(\sA_{\xbar}/\Rbar) \spt \omega_{\xbar}\big).
	\end{equation*}
     Write $\MbarO:=\ker (\VVbar)=\Im(\bar \FF)\sbt \Mbar$. Under the canonical isomorphism $\Mbar\cancong \HH^1_{\dR}(\sA_{\xbar}/\Rbar)$, $\MbarO$ corresponds to the conjugate filtration $\overline{\omega}_{\xbar}$ of $\HH^1_{\dR}(\sA_{\xbar}/\Rbar)$. We also identify the canonical isomorphism
	\begin{equation*}
		\big( \MbarO \sbt \Mbar \big) \cancong \big(  \overline{\omega}_{\xbar}\sbt \HH^1_{\dR}(\sA_{\xbar}/\Rbar)\big).
	\end{equation*}
	With these identifications, the zip isomorphism \eqref{ZipIso1} is nothing but the pull back to $\Rbar$ along $\xbar$ of the zip isomorphism  \eqref{ZIsomDeRham}. In what follows we write $\updelta$ in this form:
	\begin{equation}\label{ZipIso2}
		\updelta: \MbarIsig\oplus \Mbarsig/\MbarIsig\xrightarrow{\  [\VVbar]^{-1}\oplus \FFbar \ } \Mbar/\MbarO \oplus \MbarO.   
	\end{equation}
	
	\subsection{Comparison of $ \upeta $ and $ \upzeta $}
	Under the isomorphism $ \Gzips \cong \zipstack $, the universal $ G $-zip  in \S \ref{SectionDefinitonofEO}, $ \underI$,  corresponds to an $ \emu $-torsor $ \sZ=\sZzeta $ over $ S $, together with an $ \emu $-equivariant map $ \tilde{\upzeta}: \sZ\to \Gka $. The $ \emu  $-torsor $ \sZ $ is given by the pull-back of the canonical projection $\rI_-\to \rI_{-}/\Uminsig$ of $ S $-morphism along the $S$-morphism  																						
	\[
	\iplus\xrightarrow{\ \sig\ } \Iplusig \xrightarrow{\ \ \ \ } \Iplusig/\Uplusig\xrightarrow{\ \iota\ } \Imin/\Uminsig,
	\]
	The map $ \zetatil: \sZ\to \Gka$ is given by sending a local section $ (\xbarflat, \xbarleftflat)$ of  $\sZ\sbt \iplus\times_S \Imin$ to 
	\[
	\zetatil(\xbarflat, \xbarleftflat):= \betxbar^{-1}\circ\thetxbar,
	\] 
	which is a local section of $ G\sbt  \GL(\Lmd^*)$. Here $\xbarleftflat=(\xbar, \teta_{\xbar})$ is a local section of $ \rI_{-} $, with the same underling point $\xbar$ as that of $\xbarflat$. 
	On the other hand, under the isomorphism $[(\FlG)/P_+]\cong \zipstack$, the $P_+$-equivariant map corresponds to an $\emu$-torsor $\sZeta $ over $S$, given by the pull back of $\gam: \rI_+\to \FlG$ along the canonical projection $\Gka \to \FlG$, together with an $\emu$-equivariant map
	$\etatil: \sZeta\to \Gka$ given by the canonical projection from $\sZeta$ to $\Gka$.  
	The right $\emu$-action on $\sZeta$ is given by 
	\[
	(\xbarflat, g)\cdot (\pplu, \pmin)= (\xbarflat\cdot \pplu,\  \pplu^{-1}g \pmin).
	\]

	\begin{thm} \label{CompThm}There is a natural isomorphism $\sZ\cong \sZeta$ of  $\emu$-torsors over $S$. In other words, the two morphisms of $\kap$-algebraic stacks	$ \zet$ and $\upeta$ are $2$-isomorphic. 
	\end{thm}
	
	\begin{proof}
		
		Note that it is enough to show the commutativity of the  following diagram
		\begin{equation*}
			\xymatrixcolsep{5pc}\xymatrix{\sZ\ar[r]^{\ \zetatil \ }\ar[d]_{\pr_1}&\Gka\ar[d]\\
				\rI_+\ar[r]^{\gam}&\FlG.}
		\end{equation*}
	This is	because, once it is shown, one sees readily that the induced morphism $\sZ \to \sZeta$ of $S$-schemes, given on local sections by
		\[
		(\xbarflat, \ \xbarleftflat)\longmapsto (\xbarflat, \ \betxbar^{-1}\circ \thetxbar),
		\]
		is $\emu$-equivariant, and hence is a morphism between $\emu$-torsors over $S$, and hence is automatically an isomorphism. 

		Clearly the problem is local on $\sZ$. Let $\bar z: \Spec \bar R\to \sZ$ be an affine open of $\sZ$. In the discussion below, the underlying point $\bar x\in S(\bar R)$ is fixed. Hence to ease notation, we may write $\betxbar \in \rI_+(\bar R)$ instead of $(\xbar, \betxbar)\in \rI_+(\bar R)$; similarly for points in $\rI_-(\bar R)$. 
		We need to show the image of $\zetatil(\betxbar, \thetxbar)= \betxbar^{-1}\circ \thetxbar\in G(\bar R)$ in $\FlG(\bar R)$ coincides with $\gam(\betxbar)$.  Again since $\sZ$ is a smooth $\kap$-scheme, by Exam. \ref{ExofSimpFram} Zariski locally $\Rbar$ admits a finite $p$-basis, and hence we may choose a simple frame $(R, \sig)$ for $\Rbar$ and a lift $\xflat\in \II_+(R)$ for $\xbarflat$, we can apply the discussion in \S \ref{S:ConstructionofEta}. Then by Lem. \ref{Lem:coin}, $\gam(\betxbar)$ is equal to the image in $\FlG(\bar R)$ of 
		\[
		\smallint_{\xflat}=\betx^{-1} \Gamxflat \sig(\betx)\in \sG( R).
		\]
		Set
		\[  \teta_{x}':=\Gamxflat \sig(\betx): \big(\Lams_{ R}, \ s_{R}\big)\to  \big( M, \ s_{\dR, R}\big), \ \ \ 	\thetxbar':=\overline{\teta_{x}'}= \Gamxflatbar \sig(\betxbar).\]

		\begin{lem} \label{Lem:matching}
			We have $ (\betxbar,\  \thetxbar') \in \sZ (\Rbar)$. 
		\end{lem}
		Before showing Lem. \ref{Lem:matching}, let us note the following: it implies  Thm. \ref{CompThm}. Indeed, if Lem. \ref{Lem:matching} is shown, then
		by definition of a $G$-zip, $\thetxbar'$ and $\thetxbar$ have the same image in $\Imin/\Uminsig(\bar R)$, as they both corresponds to the image of $\betxbar$ under the isomorphism																 $\iota: \Iplusig/\Uplusig(\Rbar)\cong\Imin/\Uminsig(\Rbar)$. 						Hence we have $\thetxbar'=\thetxbar\cdot \umin$ 	for some $\umin\in \Uminsig(\bar R)$, and hence the following equality holds	\[	\overline{\smallint_{\xflat}}=\zetatil(\betxbar, \thetxbar')=\zetatil(\betxbar, \thetxbar)\umin\in G(\Rbar), \]	which  implies that the image of $\zetatil(\betxbar, \teta_{\xbar})$ in $\GmodUminsig(\Rbar)$ is equal to $\gam(\betxbar)$, as desired.
		
		\noindent
		\textbf{Proof of Lem. \ref{Lem:matching}:}
		We first show $\thetxbar'\in \Imin(\bar R)$. Note that by our discussion in  \S \ref{S:FzipandDieudonn}, the subset $\Imin(\bar R)\sbt \rI(\bar R)$ consists of elements $\thetxbar\in \rI(\bar R)$ which carries the direct summand $\Lmd^{*}_{0, \Rbar}$ of $\Lams_{\Rbar}$ isomorphically onto the conjugate filtration $\Mbar_0$ of $\bar M$.  
		
	Using notations in Lem. \ref{IntegralLem}, the normal decomposition $M=M^1\oplus M^0$ induces a decomposition $\Mbar= \Mbar^{1} \oplus \overline{M^0}$ of $\bar M$, and hence a decomposition $ \Mbarsig= \Mbar^{1,\sig} \oplus \sig^*\big(\overline{M^0}\big) $ of $ \Mbarsig $. With this decomposition, we have 	$\Mbar_0=\bar \FF\big(\sig^* \big(\overline{M^0}\big)\big)$.	From this equality we see that the direct summand of $ M $, \[M_0:=\teta_{x}'\big(\LamsO\big) =\Gamxflat\big(M^{0,\sig}\big)=\FF\big(M^{0,\sig}\big),\] is a lift of the conjugate filtration $\Mbar_{0}$ of $\bar M$ and we have  	$\thetxbar'\big(\Lams_{0,\bar R}\big)= \FF\big(\sig^*\big(\overline{ M^0}\big)\big)$.  In other words, $\thetxbar'\in \Imin(\bar R)$.  
		
		To finish the proof, we still need to show that the image of                          		$ \betxbar $ in $ \Iplusig/\Uplusig(\Rbar) $ coincides with the image of  $ \thetxbar'$ in $ \Imin/\Uminsig(\Rbar) $, via the isomorphism $ \iota:  \Iplusig/\Uplusig(\Rbar) \cong \Imin/\Uminsig(\Rbar)$. 	Denote by $ \upmu': \Gmka\xrightarrow \upmu  \Gka\hookrightarrow \GL (\Lambda_{\kap}^{*}) $ the cocharacter of   $ \GL (\Lambda_{\kap}^{*}) $ induced by $ \upmu $, as in \S \ref{S:Emdcoord}. Then we can form the $ \kap $-stack $ \GLlamstar \text{-}{\rm Zip}^{\upmu'}$. By forgetting tensors everywhere in $ \underI $, we obtain a $ \GLlamstar $-zip $  \underline{\rm I'}=(\mathrm{I'},\mathrm{I'}_{+} , \mathrm{I'}_{-}, \iota')$. 	Then by functoriality of the formation of $ G $-zips, we have the following commutative diagram 	\begin{equation*}	\xymatrix{\Iplusig/\Uplusig(\Rbar)\ar[d]\ar[r]^{\iota}&\Imin/\Uminsig(\Rbar)\ar[d]\\		\mathrm{I'}_{+}^{,\sig}/ {U'}_{+}^{, \sig}(\Rbar)\ar[r]^{\iota'}&\mathrm{I'}_{-}/{U'}_{-}^{, \sig}(\Rbar),}	\end{equation*}	where the vertical arrows are injective: this can be seen by working fppf locally and  using the fact  (see \eqref{Intersect}) \[ \mathrm{Cent}_{\GL (\Lambda_{\kap}^{*})}(\upmu')\cap \Gka= \mathrm{Cent}_{\Gka}(\upmu).\] 	Hence we are reduced to show that the image of $ \betxbar $ and $ \thetxbar' $ matches via $ \iota'$; that is, we are reduced to the case $ \Gka = \GLlamstar$.   

		Let us now unwind the definition of $\iota$ for $ \Gka = \GLlamstar$. In this special case the set $\rI_+^\sig/\Uplusig(\Rbar)$ can be realized as the set of  equivalence classes in $\rI_+(\Rbar)$ with equivalence relations given by declaring  $\bet_1, \bet_2\in \rI_+^{\sig}(\Rbar)$ equivalent if 	\[	\gr(\bet_1)=\gr(\bet_2): (\Lmd^{*}_{\Rbar} /\Lambda_{\Rbar}^{*, 1})^{\sig}\oplus 	\Lmd^{*, 1, \sig}_{\Rbar}\cong (\Mbar/\Mbar^{1}) ^{\sig}\oplus \Mbar^{1, \sig}. \]	  Similarly, the set $\rI_-/\Uminsig(\Rbar)$ can be realized as the set of equivalence classes in $\rI_-(\Rbar)$ with equivalence relations given by declaring $\thet_1, \thet_2\in \rI_-(\Rbar)$ equivalent if 
		\[
		\gr(\thet_1)=\gr(\thet_2): \Lams_{0,\Rbar}\oplus \Lams_{\Rbar} /\Lams_{0,\Rbar}\cong \Mbar_{0}\oplus \Mbar/\Mbar_{0}.
		\]
		The map $ \iota $ is given by sending the equivalence class of $ \bet \in  \mathrm{I}_{+}^{\sig}(\Rbar) $ to the unique equivalence class of $\thet\in \rI_-(\Rbar)$ with $\gr(\thet)$  equal to composition of 
		\[
		\Lams_{0,\Rbar}\oplus \Lams_{\Rbar} /\Lams_{0,\Rbar}\cong \Lams_{\Rbar}/\Lams_{-1, \Rbar}\oplus \Lams_{-1, \Rbar}\cancong (\Lmd^{*}_{\Rbar} /\Lmd^{*, 1}_{\Rbar})^{\sig}\oplus 	\Lmd^{*, 1, \sig}_{\Rbar}\xrightarrow{\bet^\sig} (\Mbar/\Mbar^{1})^ {\sig}\oplus \Mbar^{1, \sig}
		\]	
		with the zip isomorphism $(\Mbar/\Mbar^{1})^ {\sig}\oplus \Mbar^{1, \sig} \xrightarrow {\updelta} \Mbar_{0}\oplus \Mbar/\Mbar_{0}$  defined  in \eqref{ZipIso2}. Here the isomorphism $\cancong$ is induced by \eqref{ConjDecom}. Up to all kinds of identifications described above,  $\iota$ is simply given by \[\gr(\teta)\mapsto \updelta\circ \gr({\sig}(\teta)).\] Now we are reduced to verify the equality
		$\updelta=\gr\big(\overline{\Gamxflat}\big)$, which amounts to verifying the commutativity of the diagrams below, with vertical arrows canonical projections,
		\begin{equation}\label{LastDiagrm}
			\xymatrixcolsep{5pc}\xymatrix{M^{1,\sig}\oplus M^{0,\sig}\ar[d]\ar[r]^{\Gamxflat}& M_{-1}\oplus M_0\ar[d]\\
				\Mbar^{1,\sig}\oplus (\Mbar/\Mbar^1)^{\sig}\ar[r]^{\updelta}& \Mbar/\Mbar_0\oplus \Mbar_0,} 
		\end{equation}
	where $M_{-1}:=\Gamxflat(M^{1,\sig})$. 
		For the commutativity of \eqref{LastDiagrm}, we only need to check that for every element $m\in M^{1,\sig}$, we have $[\bar\VV]^{-1}(\bar m)=\overline{\Gamxflat(m)}$. 
		Note that the image $[\bar\VV]^{-1}(\bar m)$ is the unique element $\bar n\in \Mbar/\Mbar_0$ such that $\bar \VV(\bar n)=\bar m$. But $\bar \VV\big(\overline{\Gamxflat(m)}\big)=\overline{(V\circ \Gamxflat)(m)}=\bar m.$
		This finishes the proof of Lem. \ref{Lem:matching}, and hence that of Thm. \ref{CompThm}.
	\end{proof}

\end{document}